\documentclass[11pt]{article}
\usepackage{amsfonts}
\usepackage{amssymb}
\usepackage{mathrsfs}
\usepackage{soul}
\usepackage{hyperref}
\usepackage{amssymb,amsthm,amsmath,amsfonts,amsbsy,latexsym}
\usepackage{graphicx}
\usepackage[numeric,initials,nobysame]{amsrefs}
\usepackage{upref,setspace}
\usepackage{enumerate}

\numberwithin{equation}{section}
\numberwithin{figure}{section}
\numberwithin{table}{section}
\newtheorem{Lemma}{Lemma}[section]
\newtheorem{Proposition}[Lemma]{Proposition}

\newtheorem{Theorem}[Lemma]{Theorem}
\newtheorem{Assumption}[Lemma]{Assumption}

\newtheorem{Remark}[Lemma]{Remark}

\numberwithin{equation}{section}
\begin{document}

\title{Large Deviation Principle for the Exploration Process of the
Configuration Model }
\author{Shankar Bhamidi, Amarjit Budhiraja, Paul Dupuis, Ruoyu Wu}
\maketitle

\begin{abstract}
\noindent \textbf{THIS SUBMISSION HAS BEEN REPLACED WITH new url \url{https://arxiv.org/abs/1912.04714} with new results.} 

The configuration model is a sequence of random graphs constructed such that in the large network limit  the degree distribution converges to a pre-specified probability distribution. The component structure of such random graphs can be obtained from an infinite dimensional Markov chain referred to as the exploration process. We establish a large deviation principle for the exploration process associated with the configuration model. Proofs rely on a representation of the exploration process as a system of stochastic differential equations driven by Poisson random measures and variational formulas for moments of nonnegative functionals of Poisson random measures. Uniqueness results for certain controlled systems of deterministic equations play a key role in the analysis. Applications of the large deviation results, for studying asymptotic behavior of the degree sequence in large components of the random graphs, are discussed. \newline
\ \newline

\noindent

\ \newline

\end{abstract}


\section{Introduction}

\label{sec:intro} The goal of this work is to study large deviation
properties of certain random graph models. The system of interest is the
so-called \emph{configuration model } which refers to a sequence of random
graphs with number of vertices approaching infinity and the degree
distribution converging to a pre-specified probability distribution on the
set of non-negative integers \cite%
{bender1978asymptotic,bollobas1980probabilistic,molloy1995critical}. The
configuration model is a basic object in probabilistic combinatorics (cf.\ 
\cite{Hofstad2016} and references therein) and is one of the standard
workhorses in the study of networks in areas such as epidemiology \cite%
{newman2002spread} and community detection \cite%
{fortunato2010community,lancichinetti2011finding}. An important problem for
such random graph models is to estimate probabilities of non-typical
structural behaviors, particularly when the system size is large. Examples
of such behavior include, graph components that are larger or smaller than
that predicted by the law of large number analysis or degree distributions
within components that deviate significantly from their expected values.

A natural formulation of such problems of rare event probability estimation
is through the theory of large deviations. Large deviations for random graph
models has been a topic of significant recent research activity (see, e.g., \cite%
{chatterjee2011large,bordenave2015large,puhalskii2005stochastic,o1998some,choi2013large}%
). Much of the work in this area is focused on the class of \emph{dense}
random graph models (number of edges in the graph scale like $n^{2}$ where $%
n $ is the number of vertices). In this regime, the theory of graphons
obtained under dense graph limits \cite%
{borgs2008convergent,borgs2012convergent,lovasz2012large} has emerged as a
key tool in the study of large deviation asymptotics. In contrast to the
above papers, the focus in the current work is on a \emph{sparse} random
graph setting where the average degree of a typical vertex is $O(1)$ so that
the number of edges in the graph are $O(n)$ as $n\rightarrow \infty $. Here
the techniques used for the large deviations study are quite different.

The starting point of our analysis is a dynamical construction of the
configuration model given through a discrete time infinite dimensional
Markov chain referred to as the exploration process. As the name suggests,
the exploration process is constructed by first appropriately selecting a
vertex in the graph and then exploring the neighborhood of the chosen vertex
until the component of that vertex is exhausted. After this one moves on to
another `unexplored' vertex resulting in successive exploration of
components of the random graph until the entire graph has been explored. The
stochastic process corresponding to one particular coordinate of this
infinite dimensional Markov chain encodes the number of edges in any given
component through the length of its excursions away from zero. The remaining
coordinates of this Markov chain can be used to read off the number of
vertices of a given degree in any given component of the random graph. See Section \ref{sec:eea}
for a precise description of the state space of this Markov chain.
The
exploration process can be viewed as a small noise stochastic dynamical
system in which the transition steps are of size $O(1/n)$ with $n$ denoting the
number of vertices in the random graph. The main result of this work
(Theorem \ref{thm:main-ldp}) proves a large deviation principle for a
continuous time analog of the exploration process. This large deviation
principle through contraction principles  can be used to study various
asymptotic problems for the degree sequence in components of the associated random
graphs. These results are discussed in Section \ref{sec:examples}.

We now make some comments on proof techniques. The exploration process
associated with the $n$-th random graph (with $n$ vertices) in the
configuration model is described as an $\mathbb{R}^{\infty }$-valued `small
noise' Markov process $\{\boldsymbol{X}^{n}(j)\}_{j\in \mathbb{N}_{0}}$. Under our
assumptions, there exists a $N\in \mathbb{N}$ such that for all $%
j\geq nN$, $\boldsymbol{X}^{n}(j)=\boldsymbol{0}$ for all $n\in \mathbb{N}$. 
In order
to study large deviations for such a sequence, one usually considers a
sequence of continuous times processes, or equivalently $\mathbb{C}([0,N]:%
\mathbb{R}^{\infty })$-valued random variables, obtained by a linear
interpolation of $\{\boldsymbol{X}^{n}(j)\}_{j\in \mathbb{N}_{0}}$ over intervals of
length $1/n$. A large deviations analysis of such a sequence in the current
setting is challenging due to `diminishing rates' feature of the transition
kernel (see \eqref{eq:aj-dynamics}) which in turn leads to poor regularity
of the associated local rate function. By diminishing rates we mean the
property that probabilities of certain transitions, although non-zero, can
get arbitrarily close to $0$ as the system becomes large. In the model we
consider, the system will go through phases where some state transitions have
very low probabilities, that are separated by phases of `regular behavior',
many times. In terms of the underlying random graphs the first type of
phases correspond to time periods in the dynamic construction that are close
to the completion of exploration of one component and beginning of
exploration of a new component. The poor regularity of the local rate
function makes standard approximations of the near optimal trajectory that
are used in proofs of large deviation principles for such small noise
systems hard to implement. In order to overcome these difficulties we
instead consider a different continuous time process associated with the
exploration of the configuration model. This continuous time process is
obtained by introducing i.i.d.\ exponential random times before each step in the
edge exploration Markov chain. A precise description of this process is
given in terms of stochastic differential equations (SDE) driven by a
countable collection of Poisson random measures (PRM), where different PRMs
are used to describe the different types of transitions (see Section \ref%
{sec:model}). Although the coefficients in this SDE are discontinuous
functions, their dependence on the state variable is much more tractable
than the state dependence in the transition kernel of the discrete time
model.

Large deviations for small noise SDE driven by Brownian motions have been
studied extensively both in finite and infinite dimensions. An approach
based on certain variational representations for moments of nonnegative
functionals of Brownian motions and weak convergence methods \cite%
{BoueDupuis1998variational, BudhirajaDupuis2000variational} has been quite
effective in studying a broad range of such systems (cf. references in \cite{BudhirajaDupuisMaroulas2011variational}). A similar variational
representation for functionals of a Poisson random measure has been obtained
in \cite{BudhirajaDupuisMaroulas2011variational}. There have been several
recent papers that have used this representation for studying large
deviation problems (see, e.g., \cite{BudhirajaChenDupuis2013large,BudhirajaDupuisGanguly2015moderate, BudhirajaWu2017moderate}). This representation is the starting point of the analysis in  the current work as well, however the application of the representation to the setting considered here leads to some new challenges. One
key challenge that arises in the proof of the large deviations lower bound
can be described as follows. The proof of the lower bound based on
variational representations and weak convergence methods, for systems driven
by Brownian motions, requires establishing unique solvability of controlled
deterministic equations of the form 
\begin{equation}
dx(t)=b(x(t))dt+\sigma (x(t))u(t)dt,\;x(0)=x_{0},  \label{eq:condifdet}
\end{equation}%
where $u\in L^{2}([0,T]:\mathbb{R}^{d})$ (space of square integrable
functions from $[0,T]$ to $\mathbb{R}^{d}$) is a given control. It turns out
that the conditions that are typically introduced for the well-posedness of
the original small noise stochastic dynamical system of interest (e.g.\
Lipschitz properties of the coefficients $b$ and $\sigma $) are enough to
give the wellposedness of \eqref{eq:condifdet}. For example when the
coefficients are Lipschitz, one can use a standard argument based on
Gronwall's lemma and an application of the Cauchy-Schwarz inequality to
establish the desired uniqueness property. In contrast, when studying
systems driven by a PRM one instead needs to establish wellposedness of
controlled equations of the form 
\begin{equation}
x(t)=x(0)+\int_{[0,t]\times S}1_{[0,g(x(s))]}(y)\varphi
(s,y)ds\,m(dy),\;0\leq t\leq T,  \label{eq:contdetpoi}
\end{equation}%
where $S$ is a locally compact metric space, $m$ a locally finite measure on 
$S$, $g \colon \mathbb{R}\rightarrow \mathbb{R}_+$ is a measurable map and the
control $\varphi $ is a nonnegative measurable map on $[0,T]\times S$ which
satisfies the integrability property 
\begin{equation*}
\int_{\lbrack 0,T]\times S}\ell (\varphi (s,y))ds\,m(dy)<\infty ,
\end{equation*}%
where $\ell (x)=x\log x-x+1$. If $\varphi $ were uniformly bounded and $g$
sufficiently regular (e.g., Lipschitz) uniqueness follows once more by a
standard Gronwall argument. However, in general if $g$ is not Lipschitz or $%
\varphi $ is not bounded (both problems arise in the problem considered
here, see e.g. \eqref{eq:psi}-\eqref{eq:phi_k}) the problem of uniqueness
becomes a challenging obstacle. One of the novel contributions of this work
is to obtain uniqueness results for equations of the form %
\eqref{eq:contdetpoi} when certain structural properties are satisfied. The
setting we need to consider is more complex than the one described above in
that there is an infinite collection of coupled equations (one of which
corresponds to the Skorokhod problem for one dimensional reflected
trajectories) that describe the controlled system. However the basic
difficulties can already be seen for the simpler setting in %
\eqref{eq:contdetpoi}. Although for a general $\varphi $ the unique
solvability of equations of the form \eqref{eq:contdetpoi} may indeed be
intractable, the main idea in our approach is to argue that one can perturb
the original $\varphi $ slightly so that $x(\cdot )$ is the unique solution
of the corresponding equation with the perturbed $\varphi $. Furthermore the
cost difference between the original and perturbed $\varphi $ is
appropriately small. The uniqueness result given in Lemma \ref%
{lem:uniqueness} is a key ingredient in the proof of the lower bound given
in Section \ref{sec:lower}. The proof of the upper bound, via the weak
convergence based approach to large deviations relies on establishing suitable tightness and limit characterization
results for certain controlled versions of the original small noise system.
This proof is given in Section \ref{sec:upper}.

The paper is organized as follows. In Section \ref{sec:assuandres} we
introduce the configuration model, our main assumptions, and the discrete
time exploration process. We then introduce the continuous time analogue of
the exploration process. We conclude Section \ref{sec:assuandres} by introducing the rate
function and presenting our main large deviations result. In Section \ref%
{sec:examples} we comment on some applications of this large deviation
principle. In particular we deduce well know law of large number results for
the configuration model \cite{Janson2009new} and announce a result, which
will be proved in a forthcoming paper, on large deviation asymptotics for
degree distributions in components of the configuration model. Section \ref%
{sec:repnWCCP} presents the variational representation from \cite%
{BudhirajaDupuisMaroulas2011variational} for functionals of PRM that is the
starting point of our proofs. Some tightness and characterization results
that are used both in the upper and lower bound proofs are also given in
this section. Next, Section \ref{sec:upper} gives the proof of the large
deviation upper bound whereas the proof of the lower bound is given in
Section \ref{sec:lower}. Finally, Section \ref{sec:rate_function}
establishes the compactness of level sets of the function $I_{T}$ defined in
Section \ref{sec:main-result}, thus proving that $I_{T}$ is a rate function.

\subsection{Notation}

The following notation will be used. For a Polish space $\mathbb{S}$, denote
the corresponding Borel $\sigma$-field by $\mathcal{B}(\mathbb{S})$. 
Denote by $\mathcal{P}(\mathbb{S})$ (resp.\ $\mathcal{M}(\mathbb{S})$) the
space of probability measures (resp. finite measures) on $\mathbb{S}$,
equipped with the topology of weak convergence. Denote by $\mathbb{C}_b(%
\mathbb{S})$ (resp.\ $\mathbb{M}_b(\mathbb{S})$) the space of real bounded
and continuous functions (resp.\ bounded and measurable functions). For $f
\colon \mathbb{S} \to \mathbb{R}$, let $\|f\|_\infty \doteq \sup_{x \in 
\mathbb{S}} |f(x)|$. 
For a Polish space $\mathbb{S}$ and $T>0$, denote by $\mathbb{C}([0,T]:\mathbb{S})$
(resp.\ $\mathbb{D}([0,T]:\mathbb{S})$) the space of continuous functions
(resp.\ right continuous functions with left limits) from $[0,T]$ to $%
\mathbb{S}$, endowed with the uniform topology (resp.\ Skorokhod topology). 
We say a collection $\{ X^n \}$ of $\mathbb{S}$-valued random variables is
tight if the distributions of $X^n$ are tight in $\mathcal{P}(\mathbb{S})$.
A sequence of $\mathbb{D}([0,T]:\mathbb{S})$-valued random variable is said
to be $\mathcal{C}$-tight if it is tight in $\mathbb{D}([0,T]:\mathbb{S})$
and every weak limit point takes values in $\mathbb{C}([0,T]:\mathbb{S})$
a.s. We use the symbol `$\Rightarrow$' to denote convergence in
distribution. 


We denote by $\mathbb{R}^{\infty}$ the space of all real sequences which is
identified with the countable product of copies of $\mathbb{R}$. This space
is equipped with the usual product topology. For ${\boldsymbol{x}}=(x_k)_{k
\in \mathbb{N}}, {\boldsymbol{y}}=(y_k)_{k \in \mathbb{N}}$, we write ${%
\boldsymbol{x}} \le {\boldsymbol{y}}$ if $x_k \le y_k$ for each $k \in 
\mathbb{N}$. Let $\mathcal{C} \doteq \mathbb{C}([0,T]:\mathbb{R})$, $%
\mathcal{C}_\infty \doteq \mathbb{C}([0,T]:\mathbb{R}^\infty)$, $\mathcal{D}
\doteq \mathbb{D}([0,T]:\mathbb{R})$, $\mathcal{D}_\infty \doteq \mathbb{D}%
([0,T]:\mathbb{R}^\infty)$. 
Let $x^+ \doteq \max \{x,0\}$ for $x \in \mathbb{R}$.
Denote by $\mathbb{R}_+$ the set of all non-negative real numbers. 
Let $\mathbb{N}_0 \doteq \mathbb{N} \cup \{0\}$.
Cardinality of a set $A$ is denoted by $|A|$. For $n \in \mathbb{N}$, let $%
[n] \doteq \{1,2,\dotsc,n\}$.

\section{Assumptions and Results}

\label{sec:assuandres} Fix $n \in \mathbb{N}$. We start by describing the
construction of the configuration model of random graphs 
with vertex set $[n]$. Detailed description and further references for various constructions of the configuration model can be found in \cite[Chapter 7]{Hofstad2016}.

\subsection{The configuration model and assumptions}

Let ${\boldsymbol{d}}(n)=\{d_{i}^{(n)}\}_{i\in \lbrack n]}$ be a degree
sequence, namely a sequence of non-negative integers such that $%
\sum_{i=1}^{n}d_{i}^{(n)}$ is even. Let $2m^{(n)}\doteq
\sum_{i=1}^{n}d_{i}^{(n)}$. We will usually suppress the dependence of $%
d_{i}^{(n)}$ and $m^{(n)}$ on $n$ in the notation. Using the sequence $%
\{d_{i}\}$ we construct a random graph on $n$ labelled vertices $[n]$ as follows: (i)
Associate with each vertex $i\in \lbrack n]$ $d_{i}$ \emph{half-edges}. (ii)
Perform a uniform random matching on the $2m$ half-edges to form $m$ edges
so that every edge is composed of two half-edges. This procedure creates a
random multigraph $G([n],{\boldsymbol{d}}(n))$ with $m$ edges, allowing for
multiple edges and self-loops, and is called the \emph{configuration model}
with degree sequence ${\boldsymbol{d}}(n)$. Since we are concerned with
connectivity properties of the resulting graph, vertices with degree zero
play no role in our analysis, and therefore we assume that $d_{i}>0$ for all 
$i\in \lbrack n],~n\geq 1$. We make the following additional assumptions on
the collection $\{{\boldsymbol{d}}(n),n\in \mathbb{N}\}$.

\begin{Assumption}
\label{asp:convgN} There exists a probability distribution ${\boldsymbol{p}}%
\doteq \left\{ p_{k}\right\} _{k\in \mathbb{N}}$ on $\mathbb{N}$ such that,
writing $n_{k}^{\scriptscriptstyle(n)}\doteq |\left\{ i\in \lbrack n]:d_{i}=k\right\} |$ for the number of vertices with degree $%
k $, 
\begin{equation*}
\frac{n_{k}^{\scriptscriptstyle(n)}}{n}\rightarrow p_{k}\mbox{ as }%
n\rightarrow \infty ,\mbox{ for all }k\in \mathbb{N}.
\end{equation*}
\end{Assumption}

We will also usually suppress the dependence of $n_{k}^{\scriptscriptstyle(n)}$ on $n$ in the notation.
We make the following assumption on moments of the degree distribution.

\begin{Assumption}
\label{asp:exponential-boundN} There exists some $\varepsilon_{\boldsymbol{p}%
} \in (0,\infty)$ such that $\sup_{n \in \mathbb{N}} \sum_{k=1}^{\infty} 
\frac{n_k}{n}k^{1+\varepsilon_{\boldsymbol{p}}} < \infty$. 
\end{Assumption}

The above two assumptions will be made throughout this work.

\begin{Remark}
\label{rem1.1}
\begin{enumerate}[\upshape (i)]
\item Note that Assumptions \ref{asp:convgN} and \ref{asp:exponential-boundN}%
, along with Fatou's lemma, imply that $\sum_{k=1}^{\infty
}p_{k}k^{1+\varepsilon _{\boldsymbol{p}}}<\infty $. Conversely, if $%
\sum_{k=1}^{\infty }p_{k}k^{\lambda }<\infty $ for some $\lambda \in
(4,\infty )$ and $\{D_{i}\}_{i\in \mathbb{N}}$ is a sequence of i.i.d.\ $%
\mathbb{N}$-valued random variables with common distribution $%
\{p_{k}\}_{k\in \mathbb{N}}$, then using a Borel--Cantelli argument it can be
shown that for a.e.\ $\omega $, Assumptions \ref{asp:convgN} and \ref%
{asp:exponential-boundN} are satisfied with $d_{i}=D_{i}(\omega )$, $i\in
\lbrack n]$, $n\in \mathbb{N}$, and $\varepsilon _{\boldsymbol{p}}=\frac{%
\lambda }{4}-1$.
\item Under Assumptions \ref{asp:convgN} and \ref{asp:exponential-boundN}, $%
\mu\doteq\sum_{k=1}^\infty kp_k < \infty$ and the total number of edges $m = \frac{1}{2} \sum_{i=1}^n d_i$ satisfies $\frac{m}{n} \to \frac{1}{2} \sum_{k=1}^\infty kp_k$ as $n \to \infty$.
\end{enumerate}
\end{Remark}

\subsection{Edge-exploration algorithm (EEA)}

\label{sec:eea}

One can construct $G([n],{\boldsymbol{d}}(n))$ whilst simultaneously
exploring its component structure \cite{Janson2009new}, which we now describe.
This algorithm traverses the graph by exploring all its edges, unlike
typical graph exploration algorithms, which sequentially explore vertices.
At each stage of the algorithm, every vertex in $[n]$ is in one of two possible
states, sleeping or awake, while each half-edge is in one of three states:
sleeping (unexplored), active or dead (removed). Write $\mathcal{A}_{\mathbb{%
V}}(j),\mathcal{S}_{\mathbb{V}}(j)$ for the set of active and sleeping
vertices at step $j$ and similarly let $\mathcal{S}_{\mathbb{E}}(j),\mathcal{%
A}_{\mathbb{E}}(j),\mathcal{D}_{\mathbb{E}}(j)$ be the set of sleeping,
active and dead half-edges at step $j$. We call a half-edge
\textquotedblleft living\textquotedblright\ if it is either sleeping or
active. Initialize by setting all vertices and half-edges to be in the
sleeping state. For step $j\geq 0$, write $A(j)\doteq |\mathcal{A}_{\mathbb{E%
}}(j)|$ for the number of active half-edges and $V_{k}(j)$ for the number of
sleeping vertices $v\in \mathcal{S}_{\mathbb{V}}(j)$ with degree $k$. Write $\boldsymbol{V}(j)\doteq(V_{k}(j), k\in \mathbb{N})$
for the corresponding vector in $\mathbb{R}_{+}^{\infty }$. 

At step $j=0$, all vertices and half-edges are asleep hence $A(0)=0$ and $%
V_k(0)=n_k$ for $k \ge 1$. The exploration process proceeds as follows:

\begin{enumerate}[\upshape(1)]

\item If the number of active half-edges and sleeping vertices is zero, i.e. $A(j) = 0$
and $\boldsymbol{V}(j)={\boldsymbol{0}}$, all vertices and half-edges have been explored
and we terminate the algorithm.

\item If $A(j) = 0$ and $\boldsymbol{V}(j) \neq {\boldsymbol{0}}$, so there exist
sleeping vertices, pick one such vertex with probability proportional to its
degree (alternatively pick a sleeping half-edge uniformly at random) and
mark the vertex as awake and all its half-edges as active. Thus the
transition $(A(j),\boldsymbol{V}(j))$ to $(A(j+1),\boldsymbol{V}(j+1))$ at step $j+1$ takes the form 
\begin{equation*}
(0, {\boldsymbol{v}}) \mapsto (k, {\boldsymbol{v}} - {\boldsymbol{e}}_k) %
\mbox{ with probability } \frac{kv_k}{\sum_{i=1}^\infty i v_i}, \: k \in 
\mathbb{N},
\end{equation*}
where ${\boldsymbol{e}}_k$ is the $k$-th unit vector.

\item If $A(j) > 0$, pick an active half-edge uniformly at random, pair it
with another uniformly chosen living half-edge (either active or sleeping),
say $e^*$, merge both half-edges to form a full edge and kill both
half-edges. If $e^*$ was sleeping when picked, wake the vertex corresponding
to the half-edge $e^*$, and mark all its other half-edges active. Thus in
this case the transition takes the form 
\begin{align*}
(a, {\boldsymbol{v}}) &\mapsto (a-2, {\boldsymbol{v}}) 
\mbox{ with
probability } \frac{a-1}{\sum_{i=1}^\infty i v_i + a-1}, \\
(a, {\boldsymbol{v}}) &\mapsto (a+k-2, {\boldsymbol{v}} - {\boldsymbol{e}}%
_k) \mbox{ with probability } \frac{kv_k}{\sum_{i=1}^\infty i v_i + a-1}, \:
k \in \mathbb{N}.
\end{align*}
\end{enumerate}

The statements in (2) and (3) can be combined as follows: If $A(j) \neq 0$
or $\boldsymbol{V}(j) \neq {\boldsymbol{0}}$, then the transition $(A(j),\boldsymbol{V}(j))$ to $%
(A(j+1),\boldsymbol{V}(j+1))$ takes the form 
\begin{align}  \label{eq:aj-dynamics}
\begin{aligned} (a, {\boldsymbol{v}}) &\mapsto (a-2, {\boldsymbol{v}})
\mbox{ with probability } \frac{(a-1)^+}{\sum_{i=1}^\infty i v_i + (a-1)^+},
\\ (a, {\boldsymbol{v}}) &\mapsto (a+k-2, {\boldsymbol{v}} -
{\boldsymbol{e}}_k) \mbox{ with probability } \frac{kv_k}{\sum_{i=1}^\infty
i v_i + (a-1)^+}, \: k \in \mathbb{N}. \end{aligned}
\end{align}

The random graph formed at the termination of the above algorithm, denoted
as $G([n],{\boldsymbol{d}}(n))$, has the same distribution as the
configuration model with degree sequence ${\boldsymbol{d}}(n)$ \cite%
{molloy1995critical,Janson2009new}.

We note that for $j> 0$, $A(j)=0$ if and only if the exploration of a
component in the random graph $G([n],{\boldsymbol{d}}(n))$ is completed at
step $j$. Thus the number of edges in a component equals the length of an
excursion of $\{A(j)\}$ away from $0$ and the largest excursion length gives
the size of the largest component, namely the number of edges in the
component with maximal number of edges. 
The vertices in each component are the vertices that are awakened during
corresponding excursions.

Note that at each step in the EEA, either a new vertex is woken up or two
half-edges are killed. Since there are a total of $n$ vertices and $2m$
half-edges, we have from Assumptions \ref{asp:convgN} and \ref%
{asp:exponential-boundN} that the algorithm terminates in at most $m+n\le n
L $ steps where $L \doteq 1 + \lfloor \sup_n \frac{1}{2}\sum_{k=1}^\infty k 
\frac{n_k}{n} \rfloor < \infty$. We define $A(j) \equiv 0$ and $\boldsymbol{V}(j)
\equiv {\boldsymbol{0}}$ for all $j \ge j_0$ where $j_0$ is the step at
which the algorithm terminates.

\subsection{An equivalent continuous time exploration process}

\label{sec:model} A natural way to study large deviation properties of the
configuration model is through the discrete time sequence $%
\{A(j),\boldsymbol{V}(j)\}_{j\in \mathbb{N}_0}$ in EEA which can be viewed as a
discrete time \textquotedblleft small noise" Markov process. In order to
study large deviations for such a sequence, a standard approach is to
consider the sequence of $\mathbb{C}([0,L]:\mathbb{R}^{\infty})$-valued
random variables obtained by a linear interpolation of $\{A(j),\boldsymbol{V}(j)\}_{j\in 
\mathbb{N}_0}$ over intervals of length $1/n$. As was noted in the
Introduction, the `diminishing rates' feature of the transition kernel %
\eqref{eq:aj-dynamics} makes the large deviations analysis of this sequence
challenging. Thus we introduce below a different continuous time process
associated with the exploration of the configuration model.

We begin by introducing some notation that will be needed to formulate the
continuous time model. For a locally compact Polish space $\mathbb{S}$, let $%
\mathcal{M}_{FC}(\mathbb{S})$ be the space of all measures $\nu$ on $(%
\mathbb{S},\mathcal{B}(\mathbb{S}))$ such that $\nu(K)<\infty$ for every
compact $K \subset \mathbb{S}$. We equip $\mathcal{M}_{FC}(\mathbb{S})$ with
the usual vague topology. This topology can be metrized such that $\mathcal{M%
}_{FC}(\mathbb{S})$ is a Polish space (see  \cite%
{BudhirajaDupuisMaroulas2011variational} for one convenient metric). A Poisson random measure (PRM) $%
N$ on a locally compact Polish space $\mathbb{S}$ with mean
measure (or intensity measure) $\nu \in \mathcal{M}_{FC}(\mathbb{S})$ is an $%
\mathcal{M}_{FC}(\mathbb{S})$-valued random variable such that for each $A
\in \mathcal{B}(\mathbb{S})$ with $\nu(A)<\infty$, $N(A)$ is
Poisson distributed with mean $\nu(A)$ and for disjoint $A_1,\dotsc,A_k \in 
\mathcal{B}(\mathbb{S})$, $N(A_1),\dotsc,N(A_k)$ are
mutually independent random variables (cf.\ \cite{IkedaWatanabe1990SDE}).

Let $(\Omega ,\mathcal{F},{P})$ be a complete probability space on which we are given
a collection of i.i.d.\ Poisson random measures $\{N%
_{k}(ds\,dy\,dz)\}_{k\in \mathbb{N}_{0}}$ on $\mathbb{R}_{+}\times \lbrack 0,1]\times 
\mathbb{R}_{+}$ with intensity measure $ds\times dy\times dz$. Define the
filtration 
\begin{equation*}
\hat{\mathcal{F}}_{t}\doteq \sigma \{N_{k}((0,s]\times A\times
B),0\leq s\leq t,A\in \mathcal{B}([0,1]),B\in \mathcal{B}(\mathbb{R}%
_{+}),k\in \mathbb{N}_{0}\},\;t\geq 0
\end{equation*}%
and let $\{\mathcal{F}_{t}\}$ be the ${P}$ completion of this
filtration. Fix $T\in (0,\infty )$. Let $\mathcal{\bar{P}}$ be the $\{%
\mathcal{F}_{t}\}_{0\leq t\leq T}$-predictable $\sigma $-field on $\Omega
\times \lbrack 0,T]$. Denote by $\bar{\mathcal{A}}_{+}$ the class of all $(%
\mathcal{\bar{P}}\otimes \mathcal{B}([0,1]))/\mathcal{B}(\mathbb{R}_{+})$%
-measurable maps from $\Omega \times \lbrack 0,T]\times \lbrack 0,1]$ to $%
\mathbb{R}_{+}$. For $\varphi \in \bar{\mathcal{A}}_{+}$, define a counting
process $N_{k}^{\varphi }$ on $[0,T]\times \lbrack 0,1]$ by 
\begin{equation*}
N_{k}^{\varphi }([0,t]\times A)\doteq \int_{\lbrack 0,t]\times A\times 
\mathbb{R}_{+}}{{1}}_{[0,\varphi (s,y)]}(z)\,N%
_{k}(ds\,dy\,dz),\:t\in \lbrack 0,T],A\in \mathcal{B}([0,1]),k\in \mathbb{N}%
_{0}.
\end{equation*}%
We think of $N_{k}^{\varphi }$ as a controlled random measure, where $%
\varphi $ is the control process that produces a thinning of the point
process $N_{k}$ in a random but non-anticipative manner to produce
a desired intensity. We will write $N_{k}^{\varphi }$ as $N_{k}^{\theta }$
if $\varphi \equiv \theta $ for some constant $\theta \in \mathbb{R}_{+}$.
Note that $N_{k}^{\theta }$ is a PRM  on $[0,T]\times [0,1]$ with intensity $\theta ds \times dy$.
For ${\boldsymbol{x}}=(x_{0},x_{1},x_{2},\dotsc )\in \mathbb{R}\times 
\mathbb{R}_{+}^{\infty }$, let 
\begin{equation}
r({\boldsymbol{x}})\doteq (x_{0})^{+}+\sum_{k=1}^{\infty }kx_{k},\quad r_{0}(%
{\boldsymbol{x}})\doteq \frac{(x_{0})^{+}}{r({\boldsymbol{x}})}{{1%
}}_{\{r({\boldsymbol{x}})>0\}},\quad r_{k}({\boldsymbol{x}})\doteq \frac{%
kx_{k}}{r({\boldsymbol{x}})}{{1}}_{\{r({\boldsymbol{x}}%
)>0\}},\quad k\in \mathbb{N}.  \label{eq:r_k}
\end{equation}%
Recall that ${\boldsymbol{e}}_{k}$ is the $k$-th unit vector in $\mathbb{R}%
^{\infty }$, $k\in \mathbb{N}_{0}$. Define the state process $%
\boldsymbol{X}^{n}(t)=(X_{0}^{n}(t),X_{1}^{n}(t),X_{2}^{n}(t),\dotsc )$ with values in $%
\mathbb{R}\times \mathbb{R}_{+}^{\infty }$ as the solution to the following
SDE: 
\begin{align*}
\boldsymbol{X}^{n}(t)=\boldsymbol{X}^{n}(0)& +\frac{1}{n}\int_{[0,t]\times \lbrack 0,1]}{{1}%
}_{\{X_{0}^{n}(s-)\geq 0\}}\left[ -2{\boldsymbol{e}}_{0}\right] {{%
1}}_{[0,r_{0}(\boldsymbol{X}^{n}(s-)))}(y)\,N_{0}^{n}(ds\,dy) \\
\quad & +\sum_{k=1}^{\infty }\frac{1}{n}\int_{[0,t]\times \lbrack 0,1]}{%
{1}}_{\{X_{0}^{n}(s-)\geq 0\}}\left[ (k-2){\boldsymbol{e}}_{0}-{%
\boldsymbol{e}}_{k}\right] {{1}}_{[0,r_{k}(\boldsymbol{X}^{n}(s-)))}(y)%
\,N_{k}^{n}(ds\,dy) \\
\quad & +\sum_{k=1}^{\infty }\frac{1}{n}\int_{[0,t]\times \lbrack 0,1]}{%
{1}}_{\{X_{0}^{n}(s-)<0\}}\left[ k{\boldsymbol{e}}_{0}-{%
\boldsymbol{e}}_{k}\right] {{1}}_{[0,r_{k}(\boldsymbol{X}^{n}(s-)))}(y)%
\,N_{k}^{n}(ds\,dy),
\end{align*}%
where $\boldsymbol{X}^{n}(0)\doteq \frac{1}{n}(-1,n_{1},n_{2},\dotsc )$. Note that the
existence and uniqueness of solutions to this SDE follows from
Assumption \ref{asp:exponential-boundN}. Here we have applied the usual
scaling to the state variable (scaled down by $\frac{1}{n}$) and time
variable (sped up by $n$). It is not difficult to see that $\frac{1}{n}%
(A(j)-1,V_{1}(j),V_{2}(j),\dotsc )$ in the discrete time EEA can be viewed
as the embedded Markov chain associated with $\boldsymbol{X}^{n}$.

We now rewrite the evolution of $\boldsymbol{X}^{n}$ as follows: 
\begin{align*}
\boldsymbol{X}^{n}(t)& =\boldsymbol{X}^{n}(0)+{\boldsymbol{e}}_{0}\sum_{k=0}^{\infty }\frac{(k-2)}{n}%
\int_{[0,t]\times \lbrack 0,1]}{{1}}_{[0,r_{k}(\boldsymbol{X}^{n}(s-)))}(y)%
\,N_{k}^{n}(ds\,dy) \\
& \quad -\sum_{k=1}^{\infty }{\boldsymbol{e}}_{k}\frac{1}{n}%
\int_{[0,t]\times \lbrack 0,1]}{{1}}_{[0,r_{k}(\boldsymbol{X}^{n}(s-)))}(y)%
\,N_{k}^{n}(ds\,dy) \\
& \quad +{\boldsymbol{e}}_{0}\sum_{k=0}^{\infty }\frac{2}{n}%
\int_{[0,t]\times \lbrack 0,1]}{{1}}_{\{X_{0}^{n}(s-)<0\}}{%
{1}}_{[0,r_{k}(\boldsymbol{X}^{n}(s-)))}(y)\,N_{k}^{n}(ds\,dy).
\end{align*}%
Here the first two integrands do not depend on the sign of $X_{0}^{n}$ and
are interpreted as the main contribution to the evolution. The last sum is a
`reflection' term in the ${\boldsymbol{e}}_{0}$ direction and makes a
contribution of $\frac{2}{n}{\boldsymbol{e}}_{0}$ only when $X_{0}^{n}(s-)<0$%
. For $t\geq 0$ define%
\begin{align}
Y^{n}(t)& \doteq X_{0}^{n}(0)+\sum_{k=0}^{\infty }\frac{k-2}{n}%
\int_{[0,t]\times \lbrack 0,1]}{{1}}_{[0,r_{k}(\boldsymbol{X}^{n}(s-)))}(y)%
\,N_{k}^{n}(ds\,dy),  \label{eq:Y_n} \\
\eta ^{n}(t)& \doteq \sum_{k=0}^{\infty }\frac{2}{n}\int_{[0,t]\times
\lbrack 0,1]}{{1}}_{\{X_{0}^{n}(s-)<0\}}{{1}}%
_{[0,r_{k}(\boldsymbol{X}^{n}(s-)))}(y)\,N_{k}^{n}(ds\,dy).  \label{eq:eta_n}
\end{align}%
Using these we can write 
\begin{align}
X_{0}^{n}(t)& =Y^{n}(t)+\eta ^{n}(t),  \label{eq:X_n_0} \\
X_{k}^{n}(t)& =X_{k}^{n}(0)-\frac{1}{n}\int_{[0,t]\times \lbrack 0,1]}{%
{1}}_{[0,r_{k}(\boldsymbol{X}^{n}(s-)))}(y)\,N_{k}^{n}(ds\,dy),\:k\in \mathbb{N}%
.  \label{eq:X_n_k}
\end{align}%
Here $\eta ^{n}$ is viewed as the regulator function which ensures that $%
X_{0}^{n}(t)\geq -\frac{1}{n}$. Note that $X_{k}^{n}(t)$ is non-increasing
and non-negative.

\subsection{Rate Function}

\label{sec:rate_def}

The main result of this work gives a large deviation principle for $\{(\boldsymbol{X}^n,
Y^n)\}_{n \in \mathbb{N}}$ in the path space $\mathcal{D}_\infty\times 
\mathcal{D}$. In this section we define the associated rate function.
Including the process $Y^n$ in the LDP is convenient for obtaining large
deviation results, for the degree distribution in giant components, of the
form given in Section \ref{sec:examples}. 

Recall the probability distribution ${\boldsymbol{p}}\doteq \{p_{k}\}_{k\in 
\mathbb{N}}$ introduced in Assumption \ref{asp:convgN}. Let $\Gamma \colon 
\mathcal{C}\rightarrow \mathcal{C}$ denote the one-dimensional Skorokhod map
defined by 
\begin{equation*}
\Gamma (\psi )(t)\doteq \psi (t)-\inf_{0\leq s\leq t}\psi (s)\wedge 0,\;t\in
\lbrack 0,T],\psi \in \mathcal{C}.
\end{equation*}%
%
%
%

Let $\mathcal{C}_T$ be the subset of $\mathcal{C}_\infty \times \mathcal{C}$%
, consisting of those functions $(\boldsymbol{\zeta},\psi)$ 
such that

\begin{enumerate}[\upshape(a)]

\item $\psi(0)=0$, and $\psi$ is absolutely continuous on $[0,T]$.

\item $\zeta_0(t) = \Gamma(\psi)(t)$ for $t \in [0,T]$.

\item For each $k \in \mathbb{N}$, $\zeta_k(0) = p_k$, $\zeta_k$ is
non-increasing and absolutely continuous and $\zeta_k(t) \ge 0$ for $t \in
[0,T]$.
\end{enumerate}

Now we define the rate function $I_T$.
For $(\boldsymbol{\zeta} ,\psi )\in (\mathcal{D}_{\infty }\times \mathcal{D})\setminus 
\mathcal{C}_{T}$, define $I_{T}(\boldsymbol{\zeta} ,\psi )\doteq \infty $. 
For $(\boldsymbol{\zeta}
,\psi )\in \mathcal{C}_{T}$, define 
\begin{equation}
I_{T}(\boldsymbol{\zeta} ,\psi )\doteq \inf_{\boldsymbol{\varphi} \in \mathcal{S}_{T}(\boldsymbol{\zeta} ,\psi
)}\left\{ \sum_{k=0}^{\infty }\int_{[0,T]\times \lbrack 0,1]}\ell (\varphi
_{k}(s,y))\,ds\,dy\right\} .  \label{eq:rate_function}
\end{equation}%
%
%
%
Here for $x\geq 0$, 
\begin{equation}
\ell (x)\doteq x\log x-x+1,  \label{eq:ell}
\end{equation}%
and the set $\mathcal{S}_{T}(\boldsymbol{\zeta} ,\psi )$ consists of all intensities $%
\boldsymbol{\varphi} =(\varphi _{k})_{k\in \mathbb{N}_{0}}$, $\varphi _{k}:[0,T]\times
\lbrack 0,1]\rightarrow \mathbb{R}_{+}$, such that 
for $t\in \lbrack 0,T]$, 
\begin{align}
\psi (t)& =\sum_{k=0}^{\infty }(k-2)\int_{[0,t]\times \lbrack 0,1]}{%
{1}}_{[0,r_{k}(\boldsymbol{\zeta} (s)))}(y)\,\varphi _{k}(s,y)ds\,dy
\label{eq:psi} \\
\zeta _{k}(t)& =p_{k}-\int_{[0,t]\times \lbrack 0,1]}{{1}}%
_{[0,r_{k}(\boldsymbol{\zeta} (s)))}(y)\,\varphi _{k}(s,y)ds\,dy,k\in \mathbb{N}.
\label{eq:phi_k}
\end{align}%
%
%
%

\begin{Remark}
\label{rmk:property_ODE} We record the following properties of a pair $(\boldsymbol{\zeta}
,\psi )\in \mathcal{C}_{T}$ that satisfies \eqref{eq:psi} and %
\eqref{eq:phi_k} for some $\boldsymbol{\varphi} \in \mathcal{S}_{T}(\boldsymbol{\zeta} ,\psi )$.

\begin{enumerate}[\upshape(a)]

\item From Assumptions \ref{asp:convgN} and \ref{asp:exponential-boundN} it follows that the
following uniform integrability holds: As $K \to \infty$, 
\begin{equation*}
\sup_{0 \le t \le T}\sum_{k=K}^\infty k \zeta_k(t) \le \sum_{k=K}^\infty k
\sup_{0 \le t \le T} \zeta_k(t) = \sum_{k=K}^\infty k p_k \to 0.
\end{equation*}
This in particular says that $r(\boldsymbol{\zeta}(\cdot)) \in \mathcal{C}$, where $%
r(\cdot)$ is defined in \eqref{eq:r_k}.

\item For any $k \in \mathbb{N}$, whenever $\zeta_k(t_k)=0$ for some $t_k \in
[0,T]$, we must have $\zeta_k(t) = 0$ for all $t \in [t_k,T]$. This is a
consequence of the fact that $\zeta_k$ is non-increasing and non-negative for
every $k$.

\item Whenever $r(\boldsymbol{\zeta}(t^*))=0$ for some $t^* \in [0,T]$, we must have from
part (b) that $\zeta_k(t) = 0$ for all $t \in [t^*,T]$ and $k \in \mathbb{N}$%
. This, together with \eqref{eq:psi}, implies that $\psi(\cdot)$ is
non-increasing on the interval $[t^*,T]$. Hence by property (b) of $\mathcal{%
C}_T$, $\zeta_0(t)$ is non-increasing and non-negative for $t \in [t^*,T]$.
Since $\zeta_0(t^*)=0$, we must then have $\zeta_0(t) = 0$ for $t \in [t^*,T]$%
, which means that $\boldsymbol{\zeta}(t) = \boldsymbol{0}$ for $t \in [t^*,T]$. Thus $%
\boldsymbol{\zeta}(t)= \boldsymbol{0}$ after the time instant 
\begin{equation}
\tau_{\boldsymbol{\zeta}} \doteq \inf \{t \in [0,T]:r(\boldsymbol{\zeta}(t))=0\} \wedge T.
\label{eq:defntauphi}
\end{equation}
\end{enumerate}
\end{Remark}

\subsection{Results}

\label{sec:main-result}

The following is the main result of this work.

\begin{Theorem}
\label{thm:main-ldp} The function $I_T$ in \eqref{eq:rate_function} is a
rate function on $\mathcal{D}_\infty \times \mathcal{D}$ and the sequence $%
\{(\boldsymbol{X}^n,Y^n)\}_{n \in \mathbb{N}}$ satisfies a large deviation principle in $%
\mathcal{D}_\infty \times \mathcal{D}$ with rate function $I_T$.
\end{Theorem}

\textbf{Outline of the proof:} Due to the equivalence between a large
deviation principle and a Laplace principle, it suffices to show the
following three statements (cf.\ \cite[Section 1.2]{DupuisEllis2011weak}).

\begin{enumerate}[\upshape(1)]

\item Laplace principle upper bound: For all $h\in \mathbb{C}_{b}(\mathcal{D}%
_{\infty }\times \mathcal{D})$,
\begin{equation}
\limsup_{n\rightarrow \infty} \frac{1}{n}\log {{E}}%
e^{-nh(\boldsymbol{X}^{n},Y^{n})}\le - \inf_{(\boldsymbol{\zeta} ,\psi )\in \mathcal{C}_{\infty }\times 
\mathcal{C}}\{I_{T}(\boldsymbol{\zeta} ,\psi )+h(\boldsymbol{\zeta} ,\psi )\}.  \label{eq:lappriupp}
\end{equation}

\item Laplace principle lower bound: For all $h \in \mathbb{C}_b(\mathcal{D}%
_\infty \times \mathcal{D})$, 
\begin{equation}  \label{eq:lapprilow}
\liminf_{n \to \infty} \frac{1}{n} \log {{E}} e^{-nh(\boldsymbol{X}^{n},Y^{n})} \ge -
\inf_{(\boldsymbol{\zeta} ,\psi ) \in \mathcal{C}_\infty\times \mathcal{C}}
\{I_T(\boldsymbol{\zeta} ,\psi )+h(\boldsymbol{\zeta} ,\psi )\}.
\end{equation}

\item $I_T$ is a rate function on $\mathcal{D}_\infty\times \mathcal{D}$:
For each $M \in [0,\infty)$, $\{(\boldsymbol{\zeta} ,\psi ) \in \mathcal{D}_\infty \times 
\mathcal{D} : I_T(\boldsymbol{\zeta} ,\psi ) \le M \}$ is a compact subset of $\mathcal{D}%
_\infty \times \mathcal{D}$.
\end{enumerate}

The first statement will be shown in Section \ref{sec:upper}, the second in
Section \ref{sec:lower} and the final statement in Section \ref{sec:rate_function}.


\section{Examples}

\label{sec:examples} In this section we discuss two examples of the use of the
large deviation principle. Proof details of the first results are provided whereas the proof of the second are given in a
forthcoming paper. 

\subsection{Law of large number limits}

\label{sec:LLN} The large deviation principle in Theorem \ref{thm:main-ldp}
can be used to identify the law of large numbers limit $(\boldsymbol{\zeta} ,\psi )$ of
the exploration process $(\boldsymbol{X}^{n},Y^{n})$, which corresponds to the \textit{unique}
pair satisfying $I_{T}(\boldsymbol{\zeta} ,\psi )=0$. In particular we recover well known
results for the asymptotics of the largest component in the configuration
model \cite{Janson2009new}. To obtain this result we assume the following
strengthened version of Assumption \ref{asp:exponential-boundN}.

\begin{Assumption}
\label{asp:exponential-boundN-S} 
$\sup_{n \in \mathbb{N}} \sum_{k=1}^{\infty} \frac{n_k}{n}k^{2} < \infty$. 
\end{Assumption}

Recall $\mu \doteq \sum_{k=1}^{\infty }kp_{k}$ and note that $\mu <\infty $.
 Define, for $z\in \lbrack 0,1]$, 
\begin{equation*}
G_{0}(z)\doteq \sum_{k=1}^{\infty }p_{k}z^{k}\;\;\mbox{ and }%
\;\;G_{1}(z)\doteq \sum_{k=1}^{\infty }\frac{kp_{k}}{\mu }z^{k-1}.
\end{equation*}%
Define $F_{s}(t)\doteq G_{0}(s)-G_{0}(st)$ for $s\in (0,1]$ and $t\in
\lbrack 0,1]$. Then $F_{s}\colon \lbrack 0,1]\rightarrow \lbrack 0,G_{0}(s)]$
is strictly decreasing and continuous. Let $F_{s}^{-1}(\cdot )$ denote the
inverse of $F_{s}$. Define 
\begin{equation*}
f_{s}(t)\doteq \left\{ 
\begin{array}{ll}
F_{s}^{-1}(t) & \mbox{ when }0\leq t\leq G_{0}(s), \\ 
0 & \mbox{ when }t>G_{0}(s).
\end{array}%
\right.
\end{equation*}%
Then $f_{s}(t)$ is strictly decreasing until it hits zero. Note that in
particular, $f_{1}(t)=F_{1}^{-1}(t){{1}}_{[0,1]}(t)$. Define $%
f_{0}(t)\doteq 0$ for $t\geq 0$.

Fix $T\geq \frac{\mu }{2}$. The following theorem together with Proposition \ref{prop:uniqueness_LLN}
characterizes the unique $(\boldsymbol{\zeta} ,\psi )\in \mathcal{C}_{T}$ that minimizes
the rate function $I_{T}(\boldsymbol{\zeta} ,\psi )$. Letting 
\begin{equation*}
\nu \doteq \frac{\sum_{k=1}^{\infty }k(k-1)p_{k}}{\sum_{k=1}^{\infty }kp_{k}}%
,
\end{equation*}%
part 1 of the theorem considers the subcritical and critical cases $\nu \leq 1$, where the size of the largest component is $o(n)$, while part 2 considers the supercritical case $\nu >1$, where the size of the largest component is $O(n)$.

\begin{Theorem}
\phantomsection
\label{thm:LLN} 
Suppose that Assumptions \ref{asp:convgN} and \ref{asp:exponential-boundN-S}
hold.

\begin{enumerate}[\upshape(1)]

\item Suppose $\sum_{k=1}^\infty k(k-2)p_k \le 0$. Define $\boldsymbol{\zeta}(t) =
(\zeta_k(t))_{k \in \mathbb{N}_0}$ and $\psi(t)$ by 
\begin{align*}
\zeta_0(t) & \doteq 0, \zeta_k(t) \doteq p_k(f_1(t))^k, k \in \mathbb{N}, \\
\psi(t) & \doteq -2 \int_0^t r_0(\boldsymbol{\zeta}(s))\,ds + \sum_{k=1}^\infty (k-2) (p_k
- \zeta_k(t)).
\end{align*}
Then $(\boldsymbol{\zeta} ,\psi ) \in \mathcal{C}_T$ and $I_T(\boldsymbol{\zeta} ,\psi )=0$.

\item Suppose $\sum_{k=1}^\infty k(k-2)p_k > 0$. If $p_1 > 0$, then there
exists a unique $\rho \in (0,1)$ such that $G_1(\rho) = \rho$. If $p_1 = 0$, 
$G_1(\rho) = \rho$ with $\rho\doteq 0$. Define $\tau = \frac{\mu}{2}%
(1-\rho^2)>0$ and define $\boldsymbol{\zeta}(t) = (\zeta_k(t))_{k \in \mathbb{N}_0}$ and $%
\psi(t)$ by 
\begin{align*}
\zeta_0(t) & \doteq \left[ \mu-2t -\mu \sqrt{1-2t/\mu}G_1 ( \sqrt{1-2t/\mu}) %
\right] {{1}}_{[0,\tau]}(t), \\
\zeta_k(t) & \doteq \left\{ 
\begin{array}{ll}
p_k(1-{2t}/{\mu})^{k/2} & \mbox{ when } 0 \le t \le \tau, \\ 
p_k \rho^k (f_\rho(t - \tau))^k & \mbox{ when } t > \tau,
\end{array}
\right. k \in \mathbb{N}, \\
\psi(t) & \doteq -2 \int_0^t r_0(\boldsymbol{\zeta}(s))\,ds + \sum_{k=1}^\infty (k-2) (p_k
- \zeta_k(t)).
\end{align*}
Then $(\boldsymbol{\zeta} ,\psi ) \in \mathcal{C}_T$ and $I_T(\boldsymbol{\zeta} ,\psi )=0$.
\end{enumerate}
\end{Theorem}

\begin{proof}	
	(1)
	Assume without loss of generality that $T\ge 1$. Since $f_1(t) \le 1$, we see from Assumption \ref{asp:exponential-boundN-S} that $r(\boldsymbol{\zeta}(\cdot))$ with $r$ from (\ref{eq:r_k}) and $\psi$ are well-defined.	
	Let $\varphi_k(s,y) = 1$ for all $k \in \mathbb{N}_0$ and $(s,y) \in [0,T] \times [0,1]$.
	It suffices to show $\boldsymbol{\varphi} \in \mathcal{S}_T(\boldsymbol{\zeta} ,\psi )$ and $(\boldsymbol{\zeta} ,\psi ) \in \mathcal{C}_T$.
	Since $f_1(t)=F^{-1}_1(t) {{1}}_{[0,1]}(t)$, we have $\tau_{\boldsymbol{\zeta}}=1$, where $\tau_{\boldsymbol{\zeta}}$ was defined in \eqref{eq:defntauphi}.
	Since $F_1(f_1(t))=t$ for $t\in [0,1]$,  we see that
	\begin{equation*}
		f'_1(t) = - \frac{1}{\sum_{k=1}^\infty k p_k (f_1(t))^{k-1}} \mbox{ for } 0 < t < \tau_{\boldsymbol{\zeta}} \mbox{ and } f'_1(t)=0 \mbox{ for } \tau_{\boldsymbol{\zeta}} < t <T.
	\end{equation*} 
	Using this it follows that for $k \in \mathbb{N}$,
	\begin{equation*}
		\zeta'_k(t) = - \frac{k \zeta_k(t)}{\sum_{j=1}^\infty j\zeta_j(t)} = - r_k(\boldsymbol{\zeta}(t)) \mbox{ for } 0 < t < \tau_{\boldsymbol{\zeta}} \mbox{ and } \zeta'_k(t)=0 \mbox{ for } \tau_{\boldsymbol{\zeta}} < t <T.
	\end{equation*}
	From this we see that \eqref{eq:phi_k} holds and we can write
	\begin{equation*}
		\psi(t) = \sum_{k=0}^\infty (k-2) \int_0^t r_k(\boldsymbol{\zeta}(s))\,ds.
	\end{equation*}
	This gives \eqref{eq:psi} and verifies that $\boldsymbol{\varphi} \in \mathcal{S}_T(\boldsymbol{\zeta} ,\psi )$.
	
	Next we argue that $(\boldsymbol{\zeta} ,\psi ) \in \mathcal{C}_T$.
	From Assumption \ref{asp:exponential-boundN-S},  for $t < \tau_{\boldsymbol{\zeta}}$, as $K \to \infty$,
	\begin{equation*}
		\sum_{k=K}^\infty |k-2| |\zeta'_k(t)| \le \sum_{k=K}^\infty k r_k(\boldsymbol{\zeta}(t)) \le \frac{\sum_{k=K}^\infty k^2 p_k}{r(\boldsymbol{\zeta}(t))} \to 0. 
	\end{equation*}
	This in particular says that $\psi$ is absolutely continuous and gives property (a) of $\mathcal{C}_T$.
	Furthermore, for $t < \tau_{\boldsymbol{\zeta}}$,
	\begin{equation*}
		\psi'(t) = \sum_{k=1}^\infty (k-2) r_k(\boldsymbol{\zeta}(t)) = \frac{\sum_{k=1}^\infty k(k-2)p_k (f_1(t))^k}{r(\boldsymbol{\zeta}(t))} \le \frac{f_1(t) \sum_{k=1}^\infty k(k-2)p_k}{r(\boldsymbol{\zeta}(t))} \le 0.
	\end{equation*}	
	Therefore $\Gamma(\psi)(t)=0=\zeta_0(t)$ for $t < \tau_{\boldsymbol{\zeta}}$.
	For $\tau_{\boldsymbol{\zeta}} \le t \le T$, clearly $\Gamma(\psi)(t)=0=\zeta_0(t)$.
	So we have checked property (b) of $\mathcal{C}_T$.
	Property (c) of $\mathcal{C}_T$ follows from the definition of $\zeta_k$, $k \in \mathbb{N}$.
	Therefore $(\boldsymbol{\zeta} ,\psi ) \in \mathcal{C}_T$ and part (1) follows.

%
%

	(2) The fact that when $p_1>0$ there is a unique $\rho \in (0,1)$ such that $G_1(\rho)=\rho$ is proved in \cite{molloy1995critical}.
	Since $f_\rho(t) \le 1$, we see from Assumption \ref{asp:exponential-boundN-S} that $r(\boldsymbol{\zeta}(\cdot))$ and $\psi$ are well-defined.
	Let $\varphi_k(s,y) = 1$ for all $k \in \mathbb{N}_0$ and $(s,y) \in [0,T] \times [0,1]$.
	Once again, it suffices to show $\boldsymbol{\varphi} \in \mathcal{S}_T(\boldsymbol{\zeta} ,\psi )$ and $(\boldsymbol{\zeta} ,\psi ) \in \mathcal{C}_T$.
	We first consider times $t < \tau$.
	Using the definitions of $r$, $G_1$ and $\tau$, for $t < \tau$
	\begin{equation*}
		r(\boldsymbol{\zeta}(t)) = \mu-2t -\mu \sqrt{1-2t/\mu}G_1 ( \sqrt{1-2t/\mu}) + \sum_{k=1}^\infty k p_k (1 - 2t/\mu)^{k/2} = \mu-2t > \mu \rho^2 \ge 0.
	\end{equation*}
	From this one can verify that for $t < \tau$, 
	\begin{align*}
		\zeta'_k(t) & = - \frac{k \zeta_k(t)}{\mu-2t} = -r_k(\boldsymbol{\zeta}(t)).
	\end{align*}
	Using this we see that \eqref{eq:phi_k} holds for $t < \tau$ and hence as before \eqref{eq:psi} holds as well.
	To show that $(\boldsymbol{\zeta} ,\psi ) \in \mathcal{C}_t$ for $t<\tau$, it suffices to show that $\psi(t)$ is absolutely continuous and $\zeta_0(t) = \psi(t)$ for $t \in [0,\tau)$.
	Note that for $t < \tau$, $\sum_{k=1}^\infty |k-2| |r_k(\boldsymbol{\zeta}(t))| \le \frac{\sum_{k=1}^\infty k^2 p_k}{\mu-2t}$.
	So we have from Assumption \ref{asp:exponential-boundN-S}  that $\psi$ is absolutely continuous over $[0,\tau]$.
	Also, one can verify that for $t < \tau$,
	\begin{equation*}
		\zeta'_0(t) = \frac{d}{dt} r(\boldsymbol{\zeta}(t)) - \sum_{k=1}^\infty k \zeta'_k(t) = - 2 + \sum_{k=1}^\infty k r_k(\boldsymbol{\zeta}(t)) = \psi'(t).
	\end{equation*}
	So $\zeta_0(t) = \psi(t)$ for $t < \tau$.
	Therefore we have shown that $\boldsymbol{\varphi} \in \mathcal{S}_t(\boldsymbol{\zeta} ,\psi )$ and $(\boldsymbol{\zeta} ,\psi ) \in \mathcal{C}_t$ for each $t<\tau$.
	
 
	We now consider $t \in [\tau, \tau_{\boldsymbol{\zeta}}]$. 
	Since $\rho \in [0,1)$ and $G_1(\rho)=\rho$, we have
	\begin{align*}
		0 & = \frac{\mu(G_1(\rho)-\rho)}{\rho-1} = \frac{1}{\rho-1} \sum_{k=1}^\infty kp_k (\rho^{k-1}-\rho) = -p_1 + \rho \sum_{k=3}^\infty kp_k \frac{\rho^{k-2}-1}{\rho-1} \\
		& = -p_1 + \rho \sum_{k=3}^\infty kp_k (\rho^{k-3}+\rho^{k-4}+\dotsb+1) \\
		& \ge -p_1 + \rho \sum_{k=3}^\infty kp_k (k-2)\rho^{k-3} \\
		& \ge \sum_{k=1}^\infty k(k-2)p_k\rho^{k-1}
	\end{align*}
	and therefore $0 \ge \sum_{k=1}^\infty k(k-2)p_k\rho^k = \sum_{k=1}^\infty k(k-2) \zeta_k(\tau)$.
	Namely, the assumption in part (1) is satisfied with ${\boldsymbol{p}}$ replaced by $\boldsymbol{\zeta}(\tau)$.
	Thus the proof for the case $t \in [\tau, \tau_{\boldsymbol{\zeta}}]$ is very similar to that in part (1), with $f_1(t)$ replaced by $f_{\rho}(t-\tau)$ and $p_k$ replaced with $\zeta_k(\tau)$, and we would like to omit the detail.
	This completes the proof of (2).
\end{proof}

The following proposition says that there is a unique $(\boldsymbol{\zeta} ,\psi )$ satisfying $I_T(\boldsymbol{\zeta} ,\psi ) = 0$, so that this pair is the law of large numbers limit.

\begin{Proposition}
\label{prop:uniqueness_LLN} Suppose Assumptions \ref{asp:convgN} and \ref%
{asp:exponential-boundN-S} hold. Then the pair $(\boldsymbol{\zeta} ,\psi )$ defined in
Theorem \ref{thm:LLN} is the unique element of $\mathcal{D}_\infty\times 
\mathcal{D}$ such that $I_T(\boldsymbol{\zeta} ,\psi ) = 0$.
\end{Proposition}

\begin{proof}
	Suppose for $i=1,2$, $(\boldsymbol{\zeta}^{(i)},\psi^{(i)})$ are two pairs such that $I_T(\boldsymbol{\zeta}^{(i)},\psi^{(i)}) = 0$.
	By the definition of $I_T(\cdot)$, $(\boldsymbol{\zeta}^{(i)},\psi^{(i)}) \in \mathcal{C}_T$.
	It will be seen in Remark \ref{rmk:unique_varphi} that there exists some $\boldsymbol{\varphi}^{(i)} \in \mathcal{S}_T(\boldsymbol{\zeta}^{(i)},\psi^{(i)})$ such that
	\begin{equation*}
		\sum_{k=0}^\infty \int_{[0,T] \times [0,1]} \ell(\varphi_k^{(i)}(s,y)) \, ds\,dy =I_T(\boldsymbol{\zeta}^{(i)},\psi^{(i)})= 0.
	\end{equation*}	
	Since $\ell(x) = 0$ if and only if $x=1$, we must have $\varphi_k^{(i)}(s,y) = 1$ for a.e.\ $(s,y) \in [0,T]\times[0,1]$ and $k \in \mathbb{N}_0$. 
	Using such $\varphi^{(i)}$ with \eqref{eq:psi} and \eqref{eq:phi_k}, we see that
	\begin{align}
		\zeta_k^{(i)}(t) & = p_k - \int_0^t  r_k(\boldsymbol{\zeta}^{(i)}(s))\,ds, k \in \mathbb{N}, \label{eq:phi_unique_LLN} \\
		\psi^{(i)}(t) & = \sum_{k=0}^\infty (k-2) \int_0^t  r_k(\boldsymbol{\zeta}^{(i)}(s))\,ds. \label{eq:psi_unique_LLN}
	\end{align}
	Since $\zeta_0^{(i)} = \Gamma(\psi^{(i)})$, we have for a.e. $t$ that $(\zeta_0^{(i)})'(t) \ge (\psi^{(i)})'(t) = \sum_{k=0}^\infty (k-2)  r_k(\boldsymbol{\zeta}^{(i)}(t))$, and by (\ref{eq:r_k})
	\begin{align*}
		\frac{d}{dt} r(\boldsymbol{\zeta}^{(i)}(t)) & = (\zeta_0^{(i)})'(t) + \sum_{k=1}^\infty k (\zeta_k^{(i)})'(t) \ge \sum_{k=0}^\infty (k-2)  r_k(\boldsymbol{\zeta}^{(i)}(t)) - \sum_{k=1}^\infty k  r_k(\boldsymbol{\zeta}^{(i)}(t)) \\
		& = -2\cdot{{1}}_{\{r(\boldsymbol{\zeta}^{(i)}(t))>0\}} \ge -2.
	\end{align*}
	Consider the strictly increasing function $g^{(i)}(t)$ defined by 
	\begin{equation}
		\label{eq:g_unique_LLN}
		g^{(i)}(0) = 0, \: (g^{(i)})'(t) = r(\boldsymbol{\zeta}^{(i)}(g^{(i)}(t))) {{1}}_{\{g^{(i)}(t) < \tau_{\boldsymbol{\zeta}^{(i)}}\}} + {{1}}_{\{g^{(i)}(t) \ge \tau_{\boldsymbol{\zeta}^{(i)}}\}},
	\end{equation}
	where $\tau_{\boldsymbol{\zeta}^{(i)}}$ is as in \eqref{eq:defntauphi}. 
	Since $\frac{d}{dt} r(\boldsymbol{\zeta}^{(i)}(t)) \in [-2,0]$ and $0 \le r(\boldsymbol{\zeta}^{(i)}(\cdot)) \le r(\boldsymbol{\zeta}^{(i)}(0)) = \sum_{k=1}^\infty kp_k < \infty$, we see that $r(\boldsymbol{\zeta}^{(i)}(\cdot))$ is bounded and Lipschitz.
	Also $r(\boldsymbol{\zeta}^{(i)}(t)) > 0$ for $t < \tau_{\boldsymbol{\zeta}^{(i)}}$.
	So we have existence and uniqueness of the strictly increasing function $g^{(i)}(t)$ before it reaches $\tau_{\boldsymbol{\zeta}^{(i)}}$.
	The existence, uniqueness and monotonicity of $g^{(i)}(t)$ after it reaches $\tau_{\boldsymbol{\zeta}^{(i)}}$ is straightforward.
	
	Now define $({\tilde{\boldsymbol{\zeta}}}^{(i)}(t),{\tilde{\psi}}^{(i)}(t)) \doteq (\boldsymbol{\zeta}^{(i)}(g^{(i)}(t)),\psi^{(i)}(g^{(i)}(t)))$.
	From \eqref{eq:phi_unique_LLN} and \eqref{eq:psi_unique_LLN} one can verify that
	\begin{align*}
		{\tilde{\zeta}}_k^{(i)}(t) & = p_k - \int_0^t  k{\tilde{\zeta}}_k^{(i)}(s)\,ds, k \in \mathbb{N}, \\
		{\tilde{\psi}}^{(i)}(t) & = \sum_{k=1}^\infty (k-2) \int_0^t  k{\tilde{\zeta}}_k^{(i)}(s)\,ds - 2 \int_0^t  {\tilde{\zeta}}_0^{(i)}(s)\,ds \\
		& = \sum_{k=1}^\infty (k-2) \int_0^t  k{\tilde{\zeta}}_k^{(i)}(s)\,ds - 2 \int_0^t  \Gamma({\tilde{\psi}}^{(i)})(s)\,ds.		
	\end{align*}
	Clearly ${\tilde{\zeta}}_k^{(1)} = {\tilde{\zeta}}_k^{(2)}$ for each $k \in \mathbb{N}$.
	Also, since $\Gamma$ is Lipschitz on path space, Gronwall's inequality implies ${\tilde{\psi}}^{(1)} = {\tilde{\psi}}^{(2)}$, and hence ${\tilde{\zeta}}_0^{(1)} = {\tilde{\zeta}}_0^{(2)}$.
	Noting that \eqref{eq:g_unique_LLN} can be written as
	\begin{equation*}
		g^{(i)}(0) = 0, \: (g^{(i)})'(t) = r(\tilde{\boldsymbol{\zeta}}^{(i)}(t)) {{1}}_{\{r(\tilde{\boldsymbol{\zeta}}^{(i)}(t)) > 0\}} + {{1}}_{\{r(\tilde{\boldsymbol{\zeta}}^{(i)}(t)) = 0\}},
	\end{equation*}
	we have $g^{(1)} = g^{(2)}$.
	Since $g^{(i)}$ is strictly increasing, its inverse function is well-defined and we must have that $(\boldsymbol{\zeta}^{(1)},\psi^{(1)})=(\boldsymbol{\zeta}^{(2)},\psi^{(2)})$.
	This completes the proof.
\end{proof}

\subsection{Large deviations for degree distributions in a component}

\label{sec:degree_distribution}
We say a component of the graph $G([n],{\boldsymbol{d}}(n))$ has degree configuration $\{\bar n_k\}_{k \in \mathbb{N}}$ if it has $\bar n_k$ vertices of degree $k$, $k \in \mathbb{N}$.

Given ${\boldsymbol{q}}=(q_{k},k\in \mathbb{N})$ such that $0\leq {\boldsymbol{q}} \leq {\boldsymbol{p}}$, 
we are interested in the asymptotic exponential rate of decay of the probability of the 
following event 
\begin{align*}
E^{n,\varepsilon }({\boldsymbol{q}})& \doteq \{\exists \text{ a component
with degree configuration }\{\bar{n}_{k}\} \\
& \quad \quad \text{ satisfying }(q_{k}-\varepsilon )n\leq \bar{n}_{k}\leq
(q_{k}+\varepsilon )n,\,k\in \mathbb{N}\}, \quad n \in \mathbb{N}, \varepsilon \in (0,1),
\end{align*}%
namely, we want to characterize $\lim_{\varepsilon \rightarrow
0}\lim_{n\rightarrow \infty }\frac{1}{n}\log {P}\left\{
E^{n,\varepsilon }({\boldsymbol{q}})\right\} $. We assume that ${\boldsymbol{%
q}}$ satisfies the following condition: 
\begin{equation*}
\sum_{k=1}^{\infty }kq_{k}>2\sum_{k=1}^{\infty }q_{k}.
\end{equation*}%
This condition guarantees that there are strictly more edges than vertices
in the component so that connected components with degree distribution ${%
\boldsymbol{q}}$ exist. Define $\beta \doteq \beta ({\boldsymbol{q}})$ as
follows: $\beta =0$ when $q_{1}=0$, and when $q_{1}>0$, $\beta \in (0,1)$ is
the unique solution (see Remark \ref{rmk:uniqueness_beta}) of the equation 
\begin{equation*}
\sum_{k=1}^{\infty }kq_{k}=(1-\beta ^{2})\sum_{k=1}^{\infty }\frac{kq_{k}}{%
1-\beta ^{k}}.
\end{equation*}%
Define the function $K({\boldsymbol{q}})$ by 
\begin{equation*}
K({\boldsymbol{q}})\doteq \left( \frac{1}{2}\sum_{k=1}^{\infty
}kq_{k}\right) \log (1-\beta ({\boldsymbol{q}})^{2})-\sum_{k=1}^{\infty
}q_{k}\log (1-\beta ({\boldsymbol{q}})^{k}).
\end{equation*}%
Further define for $\tilde{\boldsymbol{p}}\leq {\boldsymbol{p}}$, 
\begin{equation*}
H(\tilde{\boldsymbol{p}})\doteq \sum_{k=1}^{\infty }{\tilde{p}}_{k}\log {%
\tilde{p}}_{k}-\left( \frac{1}{2}\sum_{k=1}^{\infty }k{\tilde{p}}_{k}\right)
\log \left( \frac{1}{2}\sum_{k=1}^{\infty }k{\tilde{p}}_{k}\right) .
\end{equation*}%
Define ${\tilde{I}}_{1}({\boldsymbol{q}})\doteq H({\boldsymbol{q}})+H({%
\boldsymbol{p}}-{\boldsymbol{q}})-H({\boldsymbol{p}})+K({\boldsymbol{q}})$.

\begin{Remark}
\label{rmk:uniqueness_beta} The existence and uniqueness of $\beta ({%
\boldsymbol{q}})$ can be seen as follows. For $\alpha \in (0,1)$ consider%
\begin{equation*}
\alpha F(\alpha )\doteq \sum_{k=1}^{\infty }kq_{k}-(1-\alpha
^{2})\sum_{k=1}^{\infty }\frac{kq_{k}}{1-\alpha ^{k}}=\alpha \left(
\sum_{k=3}^{\infty }\frac{\alpha -\alpha ^{k-1}}{1-\alpha ^{k}}%
kq_{k}-q_{1}\right) 
\end{equation*}%
For $k\geq 3$ and $\alpha \in (0,1)$ let%
\begin{equation*}
F_{k}(\alpha )\doteq \frac{\alpha -\alpha ^{k-1}}{1-\alpha ^{k}}.
\end{equation*}%
It is easy to verify that $F_{k}(\cdot )$ is increasing on $(0,1)$. Thus for 
$\alpha \in (0,1)$, $0=F_{k}(0+)<F_{k}(\alpha )<F_{k}(1-)=\frac{k-2}{k}$,
and therefore 
\begin{equation*}
-q_{1}=F(0+)<F(\alpha )<F(1-)=\sum_{k=3}^{\infty }(k-2)q_{k}-q_{1}.
\end{equation*}%
Since $F$ is continuous on $(0,1)$, $-q_{1}<0$ and $\sum_{k=3}^{\infty
}(k-2)q_{k}-q_{1}=\sum_{k=1}^{\infty }kq_{k}-2\sum_{k=1}^{\infty }q_{k}>0$,
we have the existence and uniqueness of $\beta ({\boldsymbol{q}})$.
\end{Remark}


The following result gives asymptotics of the event $E^{n,\varepsilon }({%
\boldsymbol{q}})$. The proof of the theorem, which is based on Theorem \ref%
{thm:main-ldp}, is given in a forthcoming paper.

%

\begin{Theorem}
\label{thm:ldg_degree_distribution} 
Suppose that Assumptions \ref{asp:convgN} and \ref{asp:exponential-boundN} hold. 
Suppose ${\boldsymbol{q}} \le {\boldsymbol{p}}$ and $\sum_{k=1}^\infty k q_k > 2 \sum_{k=1}^\infty q_k$.
Then
\begin{enumerate}[(i)]

\item (Upper bound) when $p_1=p_2=0$, we have $\beta({\boldsymbol{q}})=0$, $K({%
\boldsymbol{q}})=0$ and 
\begin{equation*}
\limsup_{\varepsilon \to 0} \limsup_{n \to \infty } \frac{1}{n} \log {%
P}\left\{ E^{n,\varepsilon}({\boldsymbol{q}})\right\} \le -{%
\tilde{I}}_1({\boldsymbol{q}}).
\end{equation*}

\item (Lower bound) 
\begin{equation*}
\liminf_{\varepsilon \to 0} \liminf_{n \to \infty } \frac{1}{n} \log {%
P}\left\{ E^{n,\varepsilon}({\boldsymbol{q}})\right\} \ge -{%
\tilde{I}}_1({\boldsymbol{q}}).
\end{equation*}
\end{enumerate}
\end{Theorem}

\begin{Remark}
	Note that the upper and lower bounds in Theorem \ref{thm:ldg_degree_distribution} coincide when $p_1=p_2=0$ and give the asymptotic exponential rate of decay of the event $E^{n,\varepsilon}({\boldsymbol{q}})$.
	The requirement $p_1=p_2=0$ for the upper bound is a technical assumption to ensure that ${\tilde{I}}_1({\boldsymbol{q}})$ is the minimizer of certain optimization problem. 
\end{Remark}


\section{Representation and Weak Convergence of Controlled Processes}

\label{sec:repnWCCP} We will use the following useful representation formula
proved in \cite{BudhirajaDupuisMaroulas2011variational}. For the second
equality in the theorem see the proof of Theorem $2.4$ in \cite%
{BudhirajaChenDupuis2013large}. The representation in the cited papers is
given in terms of a single Poisson random measure with points in a locally
compact Polish space. However for the current work it is convenient to
formulate the representation in terms of a countable sequence of independent
Poisson random measures on $[0,T]\times \lbrack 0,1]$. This representation
is immediate from the results in \cite%
{BudhirajaDupuisMaroulas2011variational} and \cite%
{BudhirajaChenDupuis2013large} by viewing the countable sequence of Poisson
random measures with points in $[0,T]\times \lbrack 0,1]$ and intensity the
Lebesgue measure $\lambda _{T}$ on $[0,T]\times \lbrack 0,1]$ as a single
PRM with points in the augmented space $[0,T]\times \lbrack 0,1]\times 
\mathbb{N}_{0}$ with intensity $\lambda _{T}\otimes \varrho $, where $%
\varrho $ is the counting measure on $\mathbb{N}$. Recall that $\bar{%
\mathcal{A}}_{+}$ denotes the class of $(\mathcal{\bar{P}}\times \mathcal{B}%
([0,1]))/\mathcal{B}(\mathbb{R}_{+})$-measurable maps from $\Omega \times
\lbrack 0,T]\times \lbrack 0,1]$ to $\mathbb{R}_{+}$. For each $m\in \mathbb{%
N}$ let 
\begin{align*}
& \bar{\mathcal{A}}_{b,m}\doteq \{(\varphi _{k})_{k\in \mathbb{N}%
_{0}}:\varphi _{k}\in \bar{\mathcal{A}}_{+}\mbox{ for each }k\in \mathbb{N}%
_{0}\mbox{ such that for all }(\omega ,t,y)\in \Omega \times \lbrack
0,T]\times \lbrack 0,1], \\
& \quad \qquad \qquad \qquad \qquad \frac{1}{m}\leq \varphi _{k}(\omega
,t,y)\leq m\mbox{ for }k\leq m\mbox{ and }\varphi _{k}(\omega ,t,y)=1\mbox{
for }k>m\}
\end{align*}%
and let $\bar{\mathcal{A}}_{b}\doteq \cup _{m=1}^{\infty }\bar{\mathcal{A}}%
_{b,m}$. Recall the function $\ell $ defined in \eqref{eq:ell}.

\begin{Theorem}
\label{thm:var_repn} Let $F \in \mathbb{M}_b([\mathcal{M}_{FC}([0,T]\times[%
0,1])]^\infty)$. Then for $\theta > 0$, 
\begin{align*}
-\log {{E}} e^{-F((N_k^\theta)_{k \in \mathbb{N}_0})} & =
\inf_{\varphi_k\in \bar{\mathcal{A}}_+, k \in \mathbb{N}_0} {{E}} %
\left[ \theta \sum_{k=0}^\infty \int_{[0,T] \times [0,1]}
\ell(\varphi_k(s,y))\,ds\,dy + F((N_k^{\theta\varphi_k})_{k \in \mathbb{N}%
_0}) \right] \\
& = \inf_{\boldsymbol{\varphi} = (\varphi_k)_{k \in \mathbb{N}_0} \in \bar{\mathcal{A}}%
_b} {{E}} \left[ \theta \sum_{k=0}^\infty \int_{[0,T] \times [0,1]}
\ell(\varphi_k(s,y))\,ds\,dy + F((N_k^{\theta\varphi_k})_{k \in \mathbb{N}%
_0}) \right].
\end{align*}
\end{Theorem}

Fix $h\in \mathbb{C}_{b}(\mathcal{D}_{\infty }\times \mathcal{D})$. Since $%
(\boldsymbol{X}^{n},Y^{n})$ can be written as $\Psi ((N_{k}^{n})_{k\in \mathbb{N}_{0}})$
for some measurable function $\Psi $ from $[\mathcal{M}_{FC}([0,T]\times
\lbrack 0,1])]^{\infty }$ to $\mathcal{D}_{\infty }\times \mathcal{D}$, we
have from the second equality in Theorem \ref{thm:var_repn} that with $%
(\theta ,F)=(n,nh\circ \Psi )$, 
\begin{equation}
-\frac{1}{n}\log {{E}}e^{-nh(\boldsymbol{X}^{n},Y^{n})}=\inf_{\boldsymbol{\varphi}
^{n}=(\varphi _{k}^{n})_{k\in \mathbb{N}_{0}}\in \bar{\mathcal{A}}_{b}}{%
{E}}\left\{ \sum_{k=0}^{\infty }\int_{[0,T]\times \lbrack 0,1]}\ell
(\varphi _{k}^{n}(s,y))\,ds\,dy+h(\bar{\boldsymbol{X}}^{n},{\bar{Y}}%
^{n})\right\} .  \label{eq:mainrepn17}
\end{equation}%
Here $(\bar{\boldsymbol{X}}^{n},{\bar{Y}}^{n})=\Psi
((N_{k}^{n\varphi _{k}^{n}})_{k\in \mathbb{N}_{0}})$, which solves the
controlled analogue of \eqref{eq:Y_n}--\eqref{eq:X_n_k}, namely $\bar{\boldsymbol{X}}%
^{n}(0)\doteq \frac{1}{n}(-1,n_{1},n_{2},\dotsc )$, and for $%
t\in \lbrack 0,T]$, 
\begin{align}
{\bar{Y}}^{n}(t)& ={\bar{X}}_{0}^{n}(0)+\sum_{k=0}^{\infty }\frac{k-2}{n}\int_{[0,t]\times \lbrack 0,1]}{%
{1}}_{[0,r_{k}(\bar{\boldsymbol{X}}^{n}(s-)))}(y)\,N_{k}^{n\varphi _{k}^{n}}(ds\,dy)
\label{eq:Ybar_n_upper_temp} \\
{\bar{X}}_{0}^{n}(t)& ={\bar{Y}}^{n}(t) +\frac{2}{n}\sum_{k=0}^{\infty }\int_{[0,t]\times \lbrack 0,1]}{%
{1}}_{\{{\bar{X}}_{0}^{n}(s-)<0\}}{{1}}%
_{[0,r_{k}(\bar{\boldsymbol{X}}^{n}(s-)))}(y)\,N_{k}^{n\varphi
_{k}^{n}}(ds\,dy),  \label{eq:Xbar_n_0_upper_temp} \\
{\bar{X}}_{k}^{n}(t)& ={\bar{X}}_{k}^{n}(0)-\frac{1%
}{n}\int_{[0,t]\times \lbrack 0,1]}{{1}}_{[0,r_{k}(\bar{\boldsymbol{X}}%
^{n}(s-)))}(y)\,N_{k}^{n\varphi _{k}^{n}}(ds\,dy),\; k\in \mathbb{%
N}.  \label{eq:Xbar_n_k_upper_temp}
\end{align}%
There is a bar in the notation $\bar{\boldsymbol{X}}^n, \bar{Y}^n$ (and $\bar{\nu}^n$ defined in \eqref{eq:nu_n_upper} below ) to indicate that these are `controlled' processes, given in terms of the 
control sequence $\boldsymbol{\varphi}^{n}\doteq (\varphi _{k}^{n})_{k\in \mathbb{N}_{0}}$.
We will occasionally suppress the dependence on $\boldsymbol{\varphi}^{n}$ in the notation and will make this dependence explicit if there are multiple controls (e.g. as in Section \ref{sec:upper}) 

In the proof of both the upper and lower bound it will be sufficient to
consider a sequence $\{\varphi _{k}^{n}\in \bar{\mathcal{A}}_{+},k\in 
\mathbb{N}_{0}\}$ that satisfies the following uniform bound for some $%
M_{0}<\infty $ 
\begin{equation}
\sup_{n\in \mathbb{N}}\sum_{k=0}^{\infty }\int_{[0,T]\times \lbrack
0,1]}\ell (\varphi _{k}^{n}(s,y))\,ds\,dy\leq M_{0},\mbox{ a.s. }{%
P}.  \label{eq:cost_bd_upper}
\end{equation}
In the rest of this section we study tightness and convergence properties of
controlled processes $(\bar{\boldsymbol{X}}^{n},{\bar{Y}}^{n})$
that are driven by controls $\{\varphi_k^n\}$ that satisfy the above a.s.
bound. 

From \eqref{eq:Ybar_n_upper_temp}--\eqref{eq:Xbar_n_k_upper_temp} we can
rewrite
\begin{align}
{\bar{Y}}^{n}(t)& ={\bar{X}}_{0}^{n}(0)+\sum_{k=0}^{\infty }(k-2){\bar{B}}%
_{k}^{n}(t),  \label{eq:Ybar_n_upper} \\
{\bar{X}}_{0}^{n}(t)& ={\bar{Y}}^{n}(t)+{\bar{\eta}}^{n}(t),
\label{eq:Xbar_n_0_upper} \\
{\bar{X}}_{k}^{n}(t)& ={\bar{X}}_{k}^{n}(0)-{\bar{B}}_{k}^{n}(t),k\in 
\mathbb{N},  \label{eq:Xbar_n_k_upper}
\end{align}%
where 
\begin{align}
{\bar{B}}_{k}^{n}(t)& \doteq \frac{1}{n}\int_{[0,t]\times \lbrack 0,1]}{%
{1}}_{[0,r_{k}(\bar{\boldsymbol{X}}^{n}(s-)))}(y)\,N_{k}^{n\varphi
_{k}^{n}}(ds\,dy),k\in \mathbb{N}_{0},  \label{eq:Bbar_n_k_upper} \\
{\bar{\eta}}^{n}(t)& \doteq \sum_{k=0}^{\infty }\frac{2}{n}\int_{[0,t]\times
\lbrack 0,1]}{{1}}_{\{{\bar{X}}_{0}^{n}(s-)<0\}}{{1}}%
_{[0,r_{k}(\bar{\boldsymbol{X}}^{n}(s-)))}(y)\,N_{k}^{n\varphi _{k}^{n}}(ds\,dy)  \notag
\\
& =\sum_{k=1}^{\infty }\frac{2}{n}\int_{[0,t]\times \lbrack 0,1]}{%
{1}}_{\{{\bar{X}}_{0}^{n}(s-)<0\}}{{1}}_{[0,r_{k}(\bar{\boldsymbol{X}}^{n}(s-)))}(y)\,N_{k}^{n\varphi _{k}^{n}}(ds\,dy).
\label{eq:etabar_n_upper}
\end{align}%
Here the last line follows on observing that ${{1}}_{\{{\bar{X}}_{0}^{n}(s-)<0\}}{{1}}_{[0,r_{0}(\bar{\boldsymbol{X}}^{n}(s-)))}(y)\equiv 0$.

Since $m_1 \doteq \sup_{n \in \mathbb{N}} \sum_{k=1}^\infty k \frac{n_k}{n}
< \infty$ by Assumption \ref{asp:exponential-boundN}, we have that for $t
\in [0,T]$, $-\frac{1}{n} \le {\bar{X}}^n_0(t) \le m_1$, $0 \le r(\bar{\boldsymbol{X}}%
^n(t)) \le m_1$ and $0 \le {\bar{X}}^n_k(t) \le \frac{n_k}{n}$. Also note
that both $r(\bar{\boldsymbol{X}}^n(\cdot))$ and ${\bar{X}}^n_k(\cdot)$ for $k \in 
\mathbb{N}$ are non-increasing.

The following lemma summarizes some elementary properties of $\ell$. For
part (a) we refer to \cite[Lemma 3.1]{BudhirajaDupuisGanguly2015moderate},
and part (b) is an easy calculation that is omitted.

\begin{Lemma}
\phantomsection
\label{lem:property_ell}

\begin{enumerate}[\upshape(a)] 

\item For each $\beta > 0$, there exists $\gamma(\beta) \in (0,\infty)$ such
that $\gamma(\beta) \to 0$ as $\beta \to \infty$ and $x \le \gamma(\beta)
\ell(x)$, for $x \ge \beta > 1$.

\item For $x \ge 0$, $x \le \ell(x) + 2$.
\end{enumerate}
\end{Lemma}

The following lemma proves certain uniform integrability properties for the
control sequence $\boldsymbol{\varphi}^n$.

\begin{Lemma}
\label{lem:UI_upper} Suppose that Assumptions \ref{asp:convgN} and \ref%
{asp:exponential-boundN} hold. For $K\in \mathbb{N}$ define%
\begin{equation}
{\bar{U}}_{K}\doteq \sup_{n\in \mathbb{N}}{{E}}\left\{
\sum_{k=K}^{\infty }\int_{[0,T]\times \lbrack 0,1]}k\varphi _{k}^{n}(s,y){%
{1}}_{[0,r_{k}(\bar{\boldsymbol{X}}^{n}(s)))}(y)\,ds\,dy\right\} .
\label{eq:Ubar_K}
\end{equation}%
Then ${\bar{U}}_{K}<\infty $ for each $K\in \mathbb{N}$ and $%
\lim_{K\rightarrow \infty }{\bar{U}}_{K}=0$.
\end{Lemma}

\begin{proof}
	From \eqref{eq:Bbar_n_k_upper} and \eqref{eq:Xbar_n_k_upper} it follows that
	\begin{equation*}
		{\bar{U}}_K = \sup_{n \in \mathbb{N}} {{E}} \left\{ \sum_{k=K}^\infty k {\bar{B}}^n_k(T) \right\} = \sup_{n \in \mathbb{N}} {{E}} \left\{ \sum_{k=K}^\infty k \left[ {\bar{X}}^n_k(0) - {\bar{X}}^n_k(T) \right] \right\} \le \sup_{n \in \mathbb{N}} \sum_{k=K}^\infty \frac{kn_k}{n}.
	\end{equation*}
	Recalling $\varepsilon_{\boldsymbol{p}} \in (0,\infty)$ introduced in Assumption \ref{asp:exponential-boundN}, we have
	\begin{equation*}
		\sup_{n \in \mathbb{N}} \sum_{k=K}^\infty \frac{kn_k}{n} \le 
		K^{-\varepsilon_{\boldsymbol{p}}} \sup_{n \in \mathbb{N}} \sum_{k=1}^\infty \frac{n_k}{n} k^{1+ \varepsilon_{\boldsymbol{p}}} \to 0
	\end{equation*}
	as $K \to \infty$.
	The result follows.
%
\end{proof}

The following lemma proves some key tightness properties. Write $\bar{\boldsymbol{B}}^n
\doteq ({\bar{B}}^n_k)_{k \in \mathbb{N}_0}$.
Define $\bar{\boldsymbol{\nu}}^{n}\doteq (\bar{\nu}_{k}^{n})_{k\in 
\mathbb{N}_{0}}$, where for $k\in \mathbb{N}_{0}$,
\begin{equation}
\bar{\nu} _{k}^{n}([0,t]\times A)\doteq \int_{\lbrack 0,t]\times
A}\varphi _{k}^{n}(s,y)\,ds\,dy,\quad t\in \lbrack 0,T],A\in \mathcal{B}%
([0,1]).  \label{eq:nu_n_upper}
\end{equation}%

\begin{Lemma}
\label{lem:tightness} Suppose Assumptions \ref{asp:convgN} and \ref%
{asp:exponential-boundN} hold and the bound in \eqref{eq:cost_bd_upper} is
satisfied. Then the sequence of random variables $\{(\bar{\boldsymbol{\nu}}^{n}, \bar{\boldsymbol{X}}^n, {\bar{Y}}^n, \bar{\boldsymbol{B}}^n, {\bar{\eta}}^n)\}$ is tight in $[%
\mathcal{M}([0,T]\times[0,1])]^\infty \times \mathcal{D}_\infty \times 
\mathcal{D} \times \mathcal{D}_\infty \times \mathcal{D}$.
\end{Lemma}

\begin{proof}
	We will argue the tightness of $\{\bar{\boldsymbol{\nu}}^{n}\}$ in $\mathcal{M}([0,T]\times[0,1])]^\infty$ and the $\mathcal{C}$-tightness of $\{{\bar{\boldsymbol{X}}}^n\}$, $\{{\bar{Y}}^n\}$, $\{{\bar{\boldsymbol{B}}}^n\}$, and
	$\{{\bar{\eta}}^n\}$ in $\mathcal{D}_\infty$, $\mathcal{D}$, $\mathcal{D}_\infty$, and $\mathcal{D}$ respectively.
	This will complete the proof.
	
	Consider first $\{\bar{\boldsymbol{\nu}}^{n}\}$. 
	Note that $[0,T] \times [0,1]$ is a compact metric space.
	Also from Lemma \ref{lem:property_ell}(b) and \eqref{eq:cost_bd_upper} we have a.s.\ for each $k \in \mathbb{N}_0$,
	\begin{equation*}
		\bar{\nu}^{n}_k([0,T]\times[0,1]) = \int_{[0,T] \times [0,1]} \varphi_k^n(s,y) \,ds\,dy \le \int_{[0,T] \times [0,1]} \left( \ell(\varphi_k^n(s,y))+2 \right) ds\,dy \le M_0 + 2T.
	\end{equation*}
	Hence $\{\bar{\nu}^{n}_k\}$ is tight in $\mathcal{M}([0,T]\times[0,1])$.	
	
	Next, since ${\bar{X}}^n_k(0) \in [0,1]$ for $k \in \mathbb{N}$ a.s.,  we see from \eqref{eq:Xbar_n_0_upper} and \eqref{eq:Xbar_n_k_upper} that $\mathcal{C}$-tightness of $\{{\bar{X}}^n\}$ in $\mathcal{D}_\infty$ follows once we show $\mathcal{C}$-tightness of $\{{\bar{Y}}^n\}$, $\{{\bar{\boldsymbol{B}}}^n\}$ and  $\{{\bar{\eta}}^n\}$.

	We now show that $\{({\bar{Y}}^n(t), {\bar{\boldsymbol{B}}}^n(t), {\bar{\eta}}^n(t)))\}$ is tight for each $t \in [0,T]$.
	From \eqref{eq:Ybar_n_upper}, \eqref{eq:Bbar_n_k_upper} and \eqref{eq:etabar_n_upper} we have
	\begin{align*}
		& {{E}} \left[ |{\bar{Y}}^n(t)| + \sum_{k=0}^\infty |{\bar{B}}^n_k(t)| + |{\bar{\eta}}^n(t)| \right] \\
		& \le \frac{1}{n}+ \sum_{k=0}^\infty [|k-2|+1] {{E}}|{\bar{B}}^n_k(t)| + {{E}}|{\bar{\eta}}^n(t)| \\
		& \le \frac{1}{n}+{{E}} \sum_{k=0}^\infty \int_{[0,T] \times [0,1]} [|k-2| + 1 + 2 \cdot {{1}}_{\{k\ge 1\}}] \varphi^n_k(s,y) {{1}}_{[0,r_k(\bar{\boldsymbol{X}}^n(s)))}(y) \,ds\,dy \\
		& \le \frac{1}{n}+ 3 {{E}} \int_{[0,T] \times [0,1]} \varphi^n_0(s,y) \,ds\,dy + 4{\bar{U}}_1,
	\end{align*}
	where the last line uses the definition of ${\bar{U}}_1$ in \eqref{eq:Ubar_K}.
	From Lemma \ref{lem:property_ell}(b) and \eqref{eq:cost_bd_upper}, we have
	\begin{equation*}
		{{E}} \int_{[0,T] \times [0,1]} \varphi^n_0(s,y) \,ds\,dy \le {{E}} \int_{[0,T] \times [0,1]} \left[ \ell(\varphi^n_0(s,y))+2 \right] ds\,dy \le M_0 + 2T.
	\end{equation*}
	Therefore $\sup_{n \in \mathbb{N}} {{E}} \left[ |{\bar{Y}}^n(t)| + \sum_{k=0}^\infty |{\bar{B}}^n_k(t)| + |{\bar{\eta}}^n(t)| \right] < \infty$ and we have tightness of $\{({\bar{Y}}^n(t), \bar{\boldsymbol{B}}^n(t), {\bar{\eta}}^n(t)))\}$ in $\mathbb{R} \times \mathbb{R}^\infty \times \mathbb{R}$ for each $t \in [0,T]$.
	
	We now consider fluctuations of $({\bar{Y}}^n, \bar{\boldsymbol{B}}^n, {\bar{\eta}}^n)$.
	Recall the filtration $\{\mathcal{F}_t\}_{0 \le t \le T}$.
	For $\delta \in [0,T]$, let $\mathcal{T}^\delta$ be the collection of all $[0,T-\delta]$-valued stopping times $\tau$.
	Note that for $\tau \in \mathcal{T}^\delta$,
	\begin{equation*}
		{{E}} |{\bar{Y}}^n(\tau+\delta) - {\bar{Y}}^n(\tau)| \le {{E}} \left[ \sum_{k=0}^\infty (k+2) \left| {\bar{B}}^n_k(\tau+\delta) - {\bar{B}}^n_k(\tau) \right|  \right].
	\end{equation*}
	Thus in order to argue tightness of $\{({\bar{Y}}^n, \bar{\boldsymbol{B}}^n, {\bar{\eta}}^n)\}$, by the Aldous--Kurtz tightness criterion (cf.\ \cite[Theorem 2.7]{Kurtz1981approximation}) it suffices to show that
	\begin{equation}
		\label{eq:tightness_upper}
		\limsup_{\delta \to 0} \limsup_{n \to \infty} \sup_{\tau \in \mathcal{T}^\delta} {{E}} \left[ \sum_{k=0}^\infty (k+2) \left| {\bar{B}}^n_k(\tau+\delta) - {\bar{B}}^n_k(\tau) \right| + \left| {\bar{\eta}}^n(\tau+\delta) - {\bar{\eta}}^n(\tau) \right| \right] = 0.
	\end{equation}
	From \eqref{eq:Bbar_n_k_upper} and \eqref{eq:etabar_n_upper} it follows that for every $K \in \mathbb{N}$ and $M \in (0,\infty)$,
	\begin{align*}
		& {{E}} \left[ \sum_{k=0}^\infty (k+2) \left| {\bar{B}}^n_k(\tau+\delta) - {\bar{B}}^n_k(\tau) \right| + \left| {\bar{\eta}}^n(\tau+\delta) - {\bar{\eta}}^n(\tau) \right| \right] \\
		& \le {{E}} \sum_{k=0}^\infty \int_{(\tau,\tau+\delta] \times [0,1]} (k+4) \varphi^n_k(s,y)  {{1}}_{[0,r_k(\bar{\boldsymbol{X}}^n(s)))}(y) \,ds\,dy \\
		& \le {{E}} \sum_{k=0}^{K-1} \left[ \int_{(\tau,\tau+\delta] \times [0,1]} (k+4) \varphi^n_k(s,y) {{1}}_{\{ \varphi^n_k(s,y) > M\}}  \,ds\,dy \right. \\
		& \qquad \left. + \int_{(\tau,\tau+\delta] \times [0,1]} (k+4) \varphi^n_k(s,y) {{1}}_{\{ \varphi^n_k(s,y) \le M\}}  \,ds\,dy \right] + 5 {\bar{U}}_K.
	\end{align*}
	Using Lemma \ref{lem:property_ell}(a) and \eqref{eq:cost_bd_upper}, we can bound the last display by
	\begin{align*}
		& {{E}} \sum_{k=0}^{K-1} \int_{(\tau,\tau+\delta] \times [0,1]} (K+3)\gamma(M) \ell(\varphi^n_k(s,y)) \,ds\,dy + K(K+3)M\delta + 5 {\bar{U}}_K \\
		& \quad \le (K+3)\gamma(M) M_0 + K(K+3)M\delta + 5 {\bar{U}}_K.
	\end{align*}
	Therefore
	\begin{align*}
		& \limsup_{\delta \to 0} \limsup_{n \to \infty} \sup_{\tau \in \mathcal{T}^\delta} {{E}} \left[ \sum_{k=0}^\infty (k+2) \left| {\bar{B}}^n_k(\tau+\delta) - {\bar{B}}^n_k(\tau) \right| + \left| {\bar{\eta}}^n(\tau+\delta) - {\bar{\eta}}^n(\tau) \right| \right] \\
		& \quad \le (K+3)\gamma(M) M_0 + 5 {\bar{U}}_K.
	\end{align*}
	Taking $M \to \infty$ and then $K \to \infty$, we have from Lemma \ref{lem:property_ell}(a) and Lemma \ref{lem:UI_upper} that \eqref{eq:tightness_upper} holds.
	Finally $\mathcal{C}$-tightness is an immediate consequence of the following a.s.\ bounds for any $k \in \mathbb{N}_0$ and $K \in \mathbb{N}$
	\begin{align*}
		|{\bar{B}}^n_k(t)- {\bar{B}}^n_k(t-)|  \le \frac{1}{n}, \;
		|{\bar{\eta}}^n(t)- {\bar{\eta}}^n(t-)|  \le \frac{2}{n}, \;
		|{\bar{Y}}^n_k(t)- {\bar{Y}}^n_k(t-)|  \le \frac{K}{n} + \sum_{j=K+1}^\infty \frac{jn_j}{n}
	\end{align*} 
	and Assumption \ref{asp:exponential-boundN}.
	This completes the proof.
\end{proof}

Next we will characterize weak limit points of $\{(\bar{\boldsymbol{\nu}} ^{n},\bar{\boldsymbol{X}}^{n},{\bar{Y}}^{n},\bar{\boldsymbol{B}}^{n},{\bar{\eta}}^{n})\}$. For that, we
need the following notation. For $k\in \mathbb{N}_{0}$ define the
compensated process 
\begin{equation*}
{\tilde{N}}_{k}^{n\varphi _{k}^{n}}(ds\,dy)\doteq N_{k}^{n\varphi
_{k}^{n}}(ds\,dy)-n\varphi _{k}^{n}(s,y)\,ds\,dy.
\end{equation*}%
Then ${\tilde{N}}_{k}^{n\varphi _{k}^{n}}([0,t]\times A)$ is an $\{\mathcal{F%
}_{t}\}$-martingale for $A\in \mathcal{B}([0,1])$ and $k\in \mathbb{N}_{0}$.
For $t\in \lbrack 0,T]$ and $k\in \mathbb{N}_{0}$ write%
\begin{equation}
{\bar{B}}_{k}^{n}(t)={\tilde{B}}_{k}^{n}(t)+{\hat{B}}_{k}^{n}(t),
\label{eq:Bbar_n_k_upper_decomp}
\end{equation}%
where 
\begin{equation*}
{\tilde{B}}_{k}^{n}(t)\doteq \frac{1}{n}\int_{[0,t]\times \lbrack 0,1]}{%
{1}}_{[0,r_{k}(\bar{\boldsymbol{X}}^{n}(s-)))}(y)\,{\tilde{N}}_{k}^{n\varphi
_{k}^{n}}(ds\,dy)
\end{equation*}%
is an $\{\mathcal{F}_{t}\}$-martingale and 
\begin{equation*}
{\hat{B}}_{k}^{n}(t)\doteq \int_{\lbrack 0,t]\times \lbrack 0,1]}{%
{1}}_{[0,r_{k}(\bar{\boldsymbol{X}}^{n}(s)))}(y)\varphi
_{k}^{n}(s,y)\,ds\,dy.
\end{equation*}%
Write $\tilde{\boldsymbol{B}}^n \doteq ({\tilde{B}}^n_k)_{k \in \mathbb{N}_0}$ and $\hat{\boldsymbol{B}}^n \doteq ({\hat{B}}^n_k)_{k \in \mathbb{N}_0}$.
For $t\in \lbrack 0,T]$, let $\lambda _{t}$ be the Lebesgue measure on $%
[0,t]\times \lbrack 0,1]$.

We have the following characterization of the weak limit points.

\begin{Lemma}
\label{lem:char_limit} Suppose Assumptions \ref{asp:convgN} and \ref%
{asp:exponential-boundN} hold. Also assume that the bound %
\eqref{eq:cost_bd_upper} is satisfied and suppose that $(\bar{\boldsymbol{\nu}}^{n}, \bar{\boldsymbol{X}}^n, {\bar{Y}}^n, \bar{\boldsymbol{B}}^n, {\bar{\eta}}^n)$ converges along a
subsequence, in distribution, to $(\bar{\boldsymbol{\nu}}, \bar{\boldsymbol{X}}, {\bar{Y}}, \bar{\boldsymbol{B}}, {%
\bar{\eta}}) \in [\mathcal{M}([0,T]\times[0,1])]^\infty \times \mathcal{D}%
_\infty \times \mathcal{D} \times \mathcal{D}_\infty \times \mathcal{D}$
given on some probability space $(\Omega^*,\mathcal{F}^*,{P}^*)$%
. Then the following holds ${P}^*$-a.s.

\begin{enumerate}[\upshape(a)]

\item For each $k \in \mathbb{N}_0$, $\bar{\nu}_k \ll \lambda_T$.

\item $(\bar{\boldsymbol{X}},{\bar{Y}},\bar{\boldsymbol{B}},{\bar{\eta}})\in \mathcal{C}_{\infty
}\times \mathcal{C}\times \mathcal{C}_{\infty }\times \mathcal{C}$, and for $%
t\in \lbrack 0,T]$%
\begin{align}
{\bar{X}}_{k}(t)& =p_{k}-{\bar{B}}_{k}(t)\geq 0,k\in \mathbb{N},
\label{eq:Xbar_k_upper} \\
{\bar{Y}}(t)& =\sum_{k=0}^{\infty }(k-2){\bar{B}}_{k}(t),
\label{eq:Ybar_upper} \\
{\bar{X}}_{0}(t)& ={\bar{Y}}(t)+{\bar{\eta}}(t)\geq 0.
\label{eq:Xbar_0_upper}
\end{align}

\item For $k\in \mathbb{N}_{0}$ let%
\begin{equation*}
\varphi _{k}(s,y)\doteq \frac{d\bar{\nu} _{k}}{d\lambda _{T}}(s,y),\;(s,y)\in
\lbrack 0,T]\times \lbrack 0,1].
\end{equation*}%
Then for $t\in \lbrack 0,T]$ and $k\in \mathbb{N}_{0}$%
\begin{equation}
\label{eq:Bbar_k_upper}
{\bar{B}}_{k}(t)=\int_{[0,t]\times \lbrack 0,1]}{{1}}_{[0,r_{k}(\bar{\boldsymbol{X}}(s)))}(y)\,\varphi _{k}(s,y)ds\,dy.
\end{equation}

\item ${\bar{X}}_{0}=\Gamma ({\bar{Y}})$. In particular, $(\bar{\boldsymbol{X}},{\bar{Y}%
})\in \mathcal{C}_{T}$ and $\boldsymbol{\varphi} \in \mathcal{S}_{T}(\bar{\boldsymbol{X}},{\bar{Y}})$%
.
\end{enumerate}
\end{Lemma}

\begin{proof}
	Assume without loss of generality that $(\bar{\boldsymbol{\nu}}^{n}, \bar{\boldsymbol{X}}^n, {\bar{Y}}^n, \bar{\boldsymbol{B}}^n, {\bar{\eta}}^n) \Rightarrow (\bar{\boldsymbol{\nu}}, \bar{\boldsymbol{X}}, {\bar{Y}}, \bar{\boldsymbol{B}}, {\bar{\eta}})$ along the whole sequence as $n \to \infty$. 
	
	(a) 
	This is an immediate consequence of the bound in \eqref{eq:cost_bd_upper} and Lemma A.1 of \cite{BudhirajaChenDupuis2013large}.

	(b) 
	The first statement is an immediate consequence of the $\mathcal{C}$-tightness argued in the proof of Lemma \ref{lem:tightness}. 
	Then using \eqref{eq:Xbar_n_k_upper}, Assumption \ref{asp:convgN} and the fact that ${\bar{X}}^n_k(t) \ge 0$ a.s., we have \eqref{eq:Xbar_k_upper}.
	Next, note that 
	 by Assumption \ref{asp:exponential-boundN}, as $K \to \infty$
	\begin{equation}
		\label{eq:UI_Xbar_n}
		\sup_{n \in \mathbb{N}} \sup_{0 \le t \le T} \left| \sum_{k=K}^\infty (k-2){\bar{B}}^n_k(t) \right| \le \sup_{n \in \mathbb{N}} \sum_{k=K}^\infty \frac{kn_k}{n} \le K^{-\varepsilon_{\boldsymbol{p}}} \sup_{n \in \mathbb{N}} \sum_{k=K}^\infty \frac{n_k}{n} k^{1+\varepsilon_{\boldsymbol{p}}} \to 0.
	\end{equation}
	Hence $\sum_{k=0}^\infty (k-2) {\bar{B}}^n_k \Rightarrow \sum_{k=0}^\infty (k-2) {\bar{B}}_k \in \mathcal{C}$.
	From this and \eqref{eq:Ybar_n_upper} we see that \eqref{eq:Ybar_upper} holds.
	Next, since $({\bar{Y}}^n,{\bar{\eta}}^n) \Rightarrow ({\bar{Y}},{\bar{\eta}}) \in \mathcal{C}^2$ and ${\bar{X}}^n_0(t) \ge -\frac{1}{n}$ a.s., we have from \eqref{eq:Xbar_n_0_upper} that \eqref{eq:Xbar_0_upper} holds.
	This completes the proof of part (b). 
	
	(c)
	By Doob's inequality, as $n \to \infty$
	\begin{align*}
		{{E}} \sum_{k=0}^\infty \sup_{0 \le t \le T} |{\tilde{B}}^n_k(t)|^2 
		& \le \frac{4}{n} {{E}} \sum_{k=0}^\infty \int_{[0,T] \times [0,1]} \varphi^n_k(s,y)  {{1}}_{[0,r_k(\bar{\boldsymbol{X}}^n(s)))}(y) \,ds\,dy \\
		& \le \frac{4}{n} {{E}} \sum_{k=0}^\infty \int_{[0,T] \times [0,1]} \left[ \ell(\varphi^n_k(s,y)) + 2 \right] {{1}}_{[0,r_k(\bar{\boldsymbol{X}}^n(s)))}(y) ds\,dy\\
		& \le \frac{4}{n} \left( M_0 + 2 T \right) \\
		& \to 0,
	\end{align*}
	where the second inequality follows from Lemma \ref{lem:property_ell}(b) and the third inequality follows from \eqref{eq:cost_bd_upper}.
	Therefore as $n \to \infty$
	\begin{equation}
		\label{eq:cvg_Atil_Btil_upper}
		{\tilde{\boldsymbol{B}}}^n \Rightarrow \boldsymbol{0}.
	\end{equation}	
	By appealing to the Skorokhod representation theorem, we can assume without loss of generality that $(\bar{\boldsymbol{\nu}}^{n}, \bar{\boldsymbol{X}}^n, {\bar{Y}}^n, \bar{\boldsymbol{B}}^n, {\bar{\eta}}^n, {\tilde{\boldsymbol{B}}}^n) \to (\bar{\boldsymbol{\nu}}, \bar{\boldsymbol{X}}, {\bar{Y}}, \bar{\boldsymbol{B}}, {\bar{\eta}}, \boldsymbol{0})$ a.s.\ on $(\Omega^*,\mathcal{F}^*,{P}^*)$, namely there exists some event $F \in \mathcal{F}^*$ such that ${P}^*(F^c) = 0$ and 
	\begin{equation*}
		(\bar{\boldsymbol{\nu}}^{n}, \bar{\boldsymbol{X}}^n, {\bar{Y}}^n, \bar{\boldsymbol{B}}^n, {\bar{\eta}}^n, {\tilde{\boldsymbol{B}}}^n) \to (\bar{\boldsymbol{\nu}}, \bar{\boldsymbol{X}}, {\bar{Y}}, \bar{\boldsymbol{B}}, {\bar{\eta}}, \boldsymbol{0}) \mbox{ on } F.
	\end{equation*}
	Fix ${\bar{\omega}} \in F$. The rest of the argument will be made for such an ${\bar{\omega}}$ which will be suppressed from the notation.
	From \eqref{eq:UI_Xbar_n} we have that as $n \to \infty$ $$r(\bar{\boldsymbol{X}}^n(t)) = ({\bar{X}}^n_0(t))^+ + \sum_{k=1}^\infty k {\bar{X}}^n_k(t) \to ({\bar{X}}_0(t))^+ + \sum_{k=1}^\infty k {\bar{X}}_k(t) = r(\bar{\boldsymbol{X}}(t))$$
	uniformly in $t \in [0,T]$, and $r(\bar{\boldsymbol{X}}(\cdot))$ is continuous.
	Let ${\bar{\tau}} \doteq \tau_{\bar{\boldsymbol{X}}}$, where $\tau_{\bar{\boldsymbol{X}}}$ is defined through \eqref{eq:defntauphi}, namely
	${\bar{\tau}} = \inf \{ t \in [0,T] : r(\bar{\boldsymbol{X}}(t)) = 0\} \wedge T$.
	We will argue that \eqref{eq:Bbar_k_upper} holds for all $t < {\bar{\tau}}$, $t = {\bar{\tau}}$ and $t >{\bar{\tau}}$.

	For $t < {\bar{\tau}}$, we have $r(\bar{\boldsymbol{X}}(t)) > 0$.
	Hence for each $k \in \mathbb{N}_0$,
	\begin{equation}\label{eq:cgceofindic}
		 {{1}}_{[0,r_k(\bar{\boldsymbol{X}}^n(s)))}(y) \to  {{1}}_{[0,r_k(\bar{\boldsymbol{X}}(s)))}(y)
	\end{equation}
	as $n \to \infty$ for $\lambda_t$-a.e.\ $(s,y) \in [0,t] \times [0,1]$
	since $\lambda_t\{(y,s): y = r_k(\bar{\boldsymbol{X}}(s))\}=0$.
	From \eqref{eq:cgceofindic} and the uniform integrability of $(s,y) \mapsto ({{1}}_{[0,r_k(\bar{\boldsymbol{X}}^n(s)))}(y) -{{1}}_{[0,r_k(\bar{\boldsymbol{X}}(s)))}(y))  \varphi^n_k(s,y)$ (with respect to the normalized Lebesgue measure on
	$[0,T]\times [0,1]$) which follows from \eqref{eq:cost_bd_upper} and the superlinearity of $\ell$, we have that
	$${\hat{B}}^n_k(t) - \int_{[0,t] \times [0,1]}  {{1}}_{[0,r_k(\bar{\boldsymbol{X}}(s)))}(y) \, \varphi_k^n(s,y) ds\,dy \to 0.$$
	Also, from the bound in \eqref{eq:cost_bd_upper} it follows that
	$$\int_{[0,t] \times [0,1]}  {{1}}_{[0,r_k(\bar{\boldsymbol{X}}(s)))}(y) \, \varphi_k^n(s,y) ds\,dy \to \int_{[0,t] \times [0,1]}  {{1}}_{[0,r_k(\bar{\boldsymbol{X}}(s)))}(y) \, \varphi_k(s,y) ds\,dy.$$
	Combining the two convergence statements we have
	\begin{equation}
		{\hat{B}}^n_k(t) \to \int_{[0,t] \times [0,1]}  {{1}}_{[0,r_k(\bar{\boldsymbol{X}}(s)))}(y) \, \varphi_k(s,y) ds\,dy.\label{eq:cgcebhatkt}
	\end{equation}
	The above convergence along with \eqref{eq:Bbar_n_k_upper_decomp} and \eqref{eq:cvg_Atil_Btil_upper} gives \eqref{eq:Bbar_k_upper} for $t < {\bar{\tau}}$.

	Since \eqref{eq:Bbar_k_upper} holds for $t < {\bar{\tau}}$, it also holds for $t = {\bar{\tau}}$ by continuity of $\bar{\boldsymbol{B}}$ and of the right side in \eqref{eq:Bbar_k_upper}.	
	
	Now suppose $T\ge t > {\bar{\tau}}$.
	Since $r(\bar{\boldsymbol{X}}(\cdot))$ is continuous, we see from the definition of ${\bar{\tau}}$ that $r(\bar{\boldsymbol{X}}({\bar{\tau}})) = 0$.
	Noting that $r(\bar{\boldsymbol{X}}^n(\cdot))$ is non-negative and non-increasing, so is $r(\bar{\boldsymbol{X}}(\cdot))$.
	Therefore $r(\bar{\boldsymbol{X}}(t)) = 0$ and $\bar{\boldsymbol{X}}(t) = \boldsymbol{0}$ for ${\bar{\tau}} \le t \le T$.
	From this we see that the right hand side of \eqref{eq:Bbar_k_upper} remains constant for ${\bar{\tau}} \le t \le T$ and it suffices to show that $\bar{\boldsymbol{B}}(t) = \bar{\boldsymbol{B}}({\bar{\tau}})$ for ${\bar{\tau}} < t \le T$.
	 From \eqref{eq:Bbar_n_k_upper} it follows that, for $k \in \mathbb{N}$,
	\begin{equation}
		\label{eq:cvg_upper_Bbar_k}
		\sup_{{\bar{\tau}} < t \le T} |{\bar{B}}^n_k(t) - {\bar{B}}^n_k({\bar{\tau}})| = {\bar{B}}^n_k(T) - {\bar{B}}^n_k({\bar{\tau}}) = {\bar{X}}^n_k({\bar{\tau}}) - {\bar{X}}^n_k(T) \le {\bar{X}}^n_k({\bar{\tau}}),
	\end{equation}
	which converges to ${\bar{X}}_k({\bar{\tau}}) = 0$ as $n \to \infty$.
	Hence ${\bar{B}}_k(t) = {\bar{B}}_k({\bar{\tau}})$ for ${\bar{\tau}} < t \le T$ and this gives \eqref{eq:Bbar_k_upper} for each $k \in \mathbb{N}$.
	Next we show ${\bar{B}}_0(t) = {\bar{B}}_0({\bar{\tau}})$ for ${\bar{\tau}} < t \le T$.
	From \eqref{eq:Ybar_n_upper} and \eqref{eq:Xbar_n_0_upper} we have
	\begin{align}
		& \sup_{{\bar{\tau}} < t \le T} |{\bar{B}}^n_0(t) - {\bar{B}}^n_0({\bar{\tau}})| \notag \\
		& \le \sup_{{\bar{\tau}} < t \le T} |{\bar{X}}^n_0(t) - {\bar{X}}^n_0({\bar{\tau}})| + \sup_{{\bar{\tau}} < t \le T} |{\bar{\eta}}^n(t) -{\bar{\eta}}^n({\bar{\tau}})| + \sum_{k=1}^\infty |k-2| \sup_{{\bar{\tau}} < t \le T} |{\bar{B}}^n_k(t) - {\bar{B}}^n_k({\bar{\tau}})|. \label{eq:cvg_upper_Bbar_0}
	\end{align}
	Since ${\bar{X}}^n_0(t) \ge -\frac{1}{n}$, we have
	\begin{align*}
		\sup_{{\bar{\tau}} < t \le T} |{\bar{X}}^n_0(t) - {\bar{X}}^n_0({\bar{\tau}})| & \le \sup_{{\bar{\tau}} < t \le T} |{\bar{X}}^n_0(t)| + |{\bar{X}}^n_0({\bar{\tau}})| \\
		& \le \sup_{{\bar{\tau}} < t \le T} ({\bar{X}}^n_0(t))^+ + \frac{1}{n} + ({\bar{X}}^n_0({\bar{\tau}}))^+ + \frac{1}{n} \\
		& \le \sup_{{\bar{\tau}} < t \le T} r(\bar{\boldsymbol{X}}^n(t)) + r(\bar{\boldsymbol{X}}^n({\bar{\tau}})) + \frac{2}{n} \\
		& \le 2r(\bar{\boldsymbol{X}}^n({\bar{\tau}})) + \frac{2}{n},
	\end{align*}
	where the last line follows from the fact that $r(\bar{\boldsymbol{X}}^n(t))$ is non-increasing for $t \in [0,T]$.
	From \eqref{eq:etabar_n_upper} and \eqref{eq:Bbar_n_k_upper} it follows that
	\begin{align*}
		& \sup_{{\bar{\tau}} < t \le T} |{\bar{\eta}}^n(t) -{\bar{\eta}}^n({\bar{\tau}})| \\
		& = \sup_{{\bar{\tau}} < t \le T} 2 \sum_{k=1}^\infty \frac{1}{n} \int_{({\bar{\tau}},t] \times [0,1]}  {{1}}_{\{{\bar{X}}^n_0(u-) < 0\}} {{1}}_{[0,r_k(\bar{\boldsymbol{X}}^n(u-)))}(y) \, N^{n\varphi^n_k}_k(du\,dy) \\
		& \le \sup_{{\bar{\tau}} < t \le T} 2 \sum_{k=1}^\infty \frac{1}{n} \int_{({\bar{\tau}},t] \times [0,1]}  {{1}}_{[0,r_k(\bar{\boldsymbol{X}}^n(u-)))}(y) \, N^{n\varphi^n_k}_k(du\,dy) \\
		& = \sup_{{\bar{\tau}} < t \le T} 2 \sum_{k=1}^\infty |{\bar{B}}^n_k(t) - {\bar{B}}^n_k({\bar{\tau}})|. 
	\end{align*}
	Combining above two estimates with \eqref{eq:cvg_upper_Bbar_0}, we see that as $n \to \infty$,
	\begin{align}
		\sup_{{\bar{\tau}} < t \le T} |{\bar{B}}^n_0(t) - {\bar{B}}^n_0({\bar{\tau}})| & \le 2r(\bar{\boldsymbol{X}}^n({\bar{\tau}})) + \frac{2}{n} + \sup_{{\bar{\tau}} < t \le T} \sum_{k=1}^\infty (k+4) |{\bar{B}}^n_k(t) - {\bar{B}}^n_k({\bar{\tau}})| \nonumber\\
		& \le 2r(\bar{\boldsymbol{X}}^n({\bar{\tau}})) + \frac{2}{n} + \sum_{k=1}^\infty (k+4) {\bar{X}}^n_k({\bar{\tau}}) \le 7 r(\bar{\boldsymbol{X}}^n({\bar{\tau}})) + \frac{2}{n} \label{eq:middl}\\
		& \to 7 r(\bar{\boldsymbol{X}}({\bar{\tau}})) = 0,\nonumber
	\end{align}
	where the second inequality follows from \eqref{eq:cvg_upper_Bbar_k}.
	Therefore ${\bar{B}}_0(t) = {\bar{B}}_0({\bar{\tau}})$ for ${\bar{\tau}} < t \le T$ and this gives \eqref{eq:Bbar_k_upper} for $k=0$.

	Since we have proved \eqref{eq:Bbar_k_upper} for all $t < {\bar{\tau}}$, $t = {\bar{\tau}}$ and $t >{\bar{\tau}}$, part (c) follows.

	(d)
	From \eqref{eq:Xbar_0_upper} and a well known characterization of the solution of the Skorohod problem (see, e.g., \cite[Section 3.6.C]{KaratzasShreve1991brownian}), it suffices to show that ${\bar{\eta}}(0) = 0$, ${\bar{\eta}}(t) \ge 0$, ${\bar{\eta}}(t)$ is non-decreasing for $t \in [0,T]$ and $\int_0^T {\bar{X}}_0(t) \, {\bar{\eta}}(dt) = 0$.	
	Since ${\bar{\eta}}^n(0) = 0$, ${\bar{\eta}}^n(t) \ge 0$ and ${\bar{\eta}}^n(t)$ is non-decreasing for $t \in [0,T]$, so is ${\bar{\eta}}$.
	It remains to show $\int_0^T {\bar{X}}_0(t) \, {\bar{\eta}}(dt) = 0$.
	Note that ${\bar{\eta}}^n(t)$ increases only when ${\bar{X}}^n_0(t-)<0$, namely ${\bar{X}}^n_0(t-) = - \frac{1}{n}$.
	Therefore
	\begin{equation*}
		\int_0^T \left( {\bar{X}}^n_0(t-)+\frac{1}{n} \right) {\bar{\eta}}^n(dt) = 0.
	\end{equation*}
	From this we have
	\begin{align}
		& \left| \int_0^T {\bar{X}}_0(t) \, {\bar{\eta}}(dt) \right| \notag \\
		& = \left| \int_0^T {\bar{X}}_0(t) \, {\bar{\eta}}(dt) - \int_0^T \left( {\bar{X}}^n_0(t-)+\frac{1}{n} \right) {\bar{\eta}}^n(dt) \right| \notag \\
		& \le \left| \int_0^T {\bar{X}}_0(t) \, {\bar{\eta}}(dt) - \int_0^T {\bar{X}}_0(t) \, {\bar{\eta}}^n(dt) \right| + \int_0^T |{\bar{X}}_0(t) - {\bar{X}}^n_0(t-)| \, {\bar{\eta}}^n(dt) + \frac{{\bar{\eta}}^n(T)}{n}. \label{eq:cvg_etabar_upper}
	\end{align}
	Since both ${\bar{\eta}}^n$ and ${\bar{\eta}}$ are non-decreasing, we see that ${\bar{\eta}}^n \to {\bar{\eta}}$ as finite measures on $[0,T]$.
	Combining this with the fact that ${\bar{X}}_0 \in \mathbb{C}_b([0,T]:\mathbb{R})$, we get $$\left| \int_0^T {\bar{X}}_0(t) \, {\bar{\eta}}(dt) - \int_0^T {\bar{X}}_0(t) \, {\bar{\eta}}^n(dt) \right| \to 0$$ as $n \to \infty$.
	Also from continuity of ${\bar{X}}_0$, we have uniform convergence of ${\bar{X}}^n_0$ to ${\bar{X}}_0$ and hence
	\begin{equation*}
		\int_0^T |{\bar{X}}_0(t) - {\bar{X}}^n_0(t-)| \, {\bar{\eta}}^n(dt) + \frac{{\bar{\eta}}^n(T)}{n} \le \left( \sup_{0 \le t \le T} |{\bar{X}}^n_0(t-)-{\bar{X}}_0(t)| + \frac{1}{n} \right) {\bar{\eta}}^n(T) \to 0
	\end{equation*}
	as $n \to \infty$.	
	Combining these two convergence results with \eqref{eq:cvg_etabar_upper},
	we see that
	\begin{equation*}
		\int_0^T {\bar{X}}_0(t) \, {\bar{\eta}}(dt) = 0.
	\end{equation*}
	This proves part (d)  and  completes the proof.
\end{proof}

\section{Laplace upper bound}

\label{sec:upper}

In this section we prove the Laplace upper bound \eqref{eq:lappriupp}.

From \eqref{eq:mainrepn17}, for every $n \in \mathbb{N}$, we can choose ${\tilde{\boldsymbol{\varphi}}}^n \doteq ({\tilde{\varphi}}^n_k)_{k \in \mathbb{N}_0} \in 
\bar{\mathcal{A}}_b$ such that 
\begin{equation*}
-\frac{1}{n} \log {{E}} e^{-nh(\boldsymbol{X}^{n},Y^{n})} \ge {{E}} \left\{
\sum_{k=0}^\infty \int_{[0,T] \times [0,1]} \ell({\tilde{\varphi}}_k^n(s,y))
\, ds\,dy + h({\bar{\boldsymbol{X}}}^{n,{\tilde{\boldsymbol{\varphi}}}^n}, {\bar{Y}}^{n,{\tilde{\boldsymbol{\varphi}}}^n}) \right\} - \frac{1}{n},
\end{equation*}
where $(\bar{\boldsymbol{X}}^{n,{\tilde{\boldsymbol{\varphi}}}^n}, {\bar{Y}}^{n,{\tilde{\boldsymbol{\varphi}}}%
^n}) $ are defined through \eqref{eq:Ybar_n_upper_temp}--%
\eqref{eq:Xbar_n_k_upper_temp} by replacing $\boldsymbol{\varphi}^n$ with ${\tilde{\boldsymbol{\varphi}}}%
^n$. Since $\|h\|_\infty < \infty$, we have 
\begin{equation*}
\sup_{n \in \mathbb{N}} {{E}} \sum_{k=0}^\infty \int_{[0,T] \times
[0,1]} \ell({\tilde{\varphi}}_k^n(s,y)) \, ds\,dy \le 2\|h\|_\infty + 1
\doteq M_h.
\end{equation*}
Now we modify ${\tilde{\boldsymbol{\varphi}}}^n$ so that the last inequality holds not in
the sense of expectation, but rather almost surely, for a possibly larger
constant. Fix $\sigma \in (0,1)$ and define 
\begin{equation*}
{\tilde{\tau}}^n \doteq \inf \left\{t \in [0,T] : \sum_{k=0}^\infty
\int_{[0,t] \times [0,1]} \ell({\tilde{\varphi}}_k^n(s,y)) \, ds\,dy > 2 M_h
\|h\|_\infty / \sigma \right\} \wedge T.
\end{equation*}
For $k \in \mathbb{N}_0$, letting 
\begin{equation*}
\varphi^n_k(s,y) \doteq {\tilde{\varphi}}^n_k(s,y) {{1}}_{\{s \le 
{\tilde{\tau}}^n\}} + {{1}}_{\{s > {\tilde{\tau}}^n\}}, (s,y) \in
[0,T]\times[0,1],
\end{equation*}
we have ${\boldsymbol{\varphi}}^n \doteq ({\varphi}^n_k)_{k \in \mathbb{N}_0} \in \bar{\mathcal{A}}_b$ since ${\tilde{\tau}}^n$
is an $\{\mathcal{F}_t\}$-stopping time. Also 
\begin{equation*}
{{E}} \sum_{k=0}^\infty \int_{[0,T] \times [0,1]}
\ell(\varphi_k^n(s,y)) \, ds\,dy \le {{E}} \sum_{k=0}^\infty
\int_{[0,T] \times [0,1]} \ell({\tilde{\varphi}}_k^n(s,y)) \, ds\,dy
\end{equation*}
and 
\begin{align*}
{P}(\boldsymbol{\varphi}^n \ne{\tilde{\boldsymbol{\varphi}}}^n) & \le {P}
\left( \sum_{k=0}^\infty \int_{[0,T] \times [0,1]} \ell({\tilde{\varphi}}%
_k^n(s,y)) \, ds\,dy > 2 M_h \|h\|_\infty / \sigma \right) \\
& \le \frac{\sigma}{2 M_h \|h\|_\infty} {{E}} \sum_{k=0}^\infty
\int_{[0,T] \times [0,1]} \ell({\tilde{\varphi}}_k^n(s,y)) \, ds\,dy \\
& \le \frac{\sigma}{2 \|h\|_\infty}.
\end{align*}
Letting $(\bar{\boldsymbol{X}}^{n,{\boldsymbol{\varphi}}^n}, {\bar{Y}}^{n,{\boldsymbol{\varphi}}%
^n}) $ be defined through \eqref{eq:Ybar_n_upper_temp}--%
\eqref{eq:Xbar_n_k_upper_temp} using $\boldsymbol{\varphi}^n$, we have
\begin{equation*}
\left| {{E}} h(\bar{\boldsymbol{X}}^{n,\boldsymbol{\varphi}^n},{\bar{Y}}^{n,\boldsymbol{\varphi}^n}) - {%
{E}} h(\bar{\boldsymbol{X}}^{n,{\tilde{\boldsymbol{\varphi}}}^n},{\bar{Y}}^{n,{\tilde{\boldsymbol{\varphi}}%
}^n}) \right| \le 2 \|h\|_\infty {P}(\boldsymbol{\varphi}^n \ne{\tilde{%
\boldsymbol{\varphi}}}^n) \le \sigma.
\end{equation*}
Hence we have 
\begin{equation*}
-\frac{1}{n} \log {{E}} e^{-nh(\boldsymbol{X}^{n},Y^{n})} \ge {{E}} \left\{
\sum_{k=0}^\infty \int_{[0,T] \times [0,1]} \ell(\varphi_k^n(s,y)) \, ds\,dy
+ h(\bar{\boldsymbol{X}}^{n,\boldsymbol{\varphi}^n},{\bar{Y}}^{n,\boldsymbol{\varphi}^n}) \right\} - \frac{1}{n}
- \sigma
\end{equation*}
and 
\begin{equation}  \label{eq:cost_bd_upper_B}
\sup_{n \in \mathbb{N}} \sum_{k=0}^\infty \int_{[0,T] \times [0,1]}
\ell(\varphi_k^n(s,y)) \, ds\,dy \le 2M_h \|h\|_{\infty}/\sigma \doteq K_0, %
\mbox{ a.s. } {P}.
\end{equation}

Now we can complete the proof of the Laplace upper bound. Recall that 
$h \in 
\mathbb{C}_b(\mathcal{D}_\infty \times \mathcal{D})$.
Write $(\bar{\boldsymbol{\nu}}^{n},\bar{\boldsymbol{X}}^{n},{\bar{Y}}^{n}) \doteq (\bar{\boldsymbol{\nu}}^{n,\boldsymbol{\varphi}^n}, \bar{\boldsymbol{X}}^{n,{\boldsymbol{\varphi}}^n}, {\bar{Y}}^{n,{\boldsymbol{\varphi}}%
^n}) $, where $\bar{\boldsymbol{\nu}}^{n,\boldsymbol{\varphi}^n}$ is as defined in \eqref{eq:nu_n_upper} using $\boldsymbol{\varphi}^n$.
Noting from %
\eqref{eq:cost_bd_upper_B} that \eqref{eq:cost_bd_upper} is satisfied with $%
M_0= K_0$, we have from Lemma \ref{lem:tightness} that $\{(\bar{\boldsymbol{\nu}}^{n},%
\bar{\boldsymbol{X}}^{n},{\bar{Y}}^{n})\}$ is tight. Assume without
loss of generality that $(\bar{\boldsymbol{\nu}}^{n},\bar{\boldsymbol{X}}^{n},{\bar{Y}}%
^{n})$ converges along the whole sequence weakly to $(\bar{\boldsymbol{\nu}},\bar{\boldsymbol{X}}%
,{\bar{Y}})$, given on some probability space $(\Omega^*,\mathcal{F}^*,{%
P}^*)$. By Lemma \ref{lem:char_limit} we have $(\bar{\boldsymbol{X}},{\bar{%
Y}}) \in \mathcal{C}_T$ and $\bar{\boldsymbol{\nu}} = \bar{\boldsymbol{\nu}}^{\boldsymbol{\varphi}}$ for some $\boldsymbol{\varphi} \in 
\mathcal{S}_T(\bar{\boldsymbol{X}},{\bar{Y}})$ a.s.\ ${P}^*$, where $\bar{\boldsymbol{\nu}}^{\boldsymbol{\varphi}}$ is as defined in \eqref{eq:nu_n_upper} using $\boldsymbol{\varphi}$. Using
Fatou's lemma and the definition of $I_T$ in \eqref{eq:rate_function}
\begin{align*}
\liminf_{n \to \infty} -\frac{1}{n} \log {{E}} e^{-nh(\boldsymbol{X}^{n},Y^{n})} & \ge
\liminf_{n \to \infty} {{E}} \left\{ \sum_{k=0}^\infty \int_{[0,T]
\times [0,1]} \ell(\varphi_k^n(s,y)) \, ds\,dy + h(\bar{\boldsymbol{X}}^n,{\bar{Y}}^n)
- \frac{1}{n} - \sigma \right\} \\
& \ge {{E}}^* \left\{ \sum_{k=0}^\infty \int_{[0,T] \times [0,1]}
\ell(\varphi_k(s,y)) \, ds\,dy + h(\bar{\boldsymbol{X}},{\bar{Y}}) \right\} - \sigma \\
& \ge \inf_{(\boldsymbol{\zeta} ,\psi ) \in \mathcal{D}_\infty \times \mathcal{D}}
\{I_T(\boldsymbol{\zeta} ,\psi ) + h(\boldsymbol{\zeta} ,\psi )\} - \sigma.
\end{align*}
where the second inequality is a consequence of Lemma A.1 in \cite%
{BudhirajaChenDupuis2013large}. Since $\sigma \in (0,1)$ is arbitrary, this
completes the proof of the Laplace upper bound.

\section{Laplace lower bound}

\label{sec:lower}

In this section we prove the Laplace lower bound \eqref{eq:lapprilow}.

The following lemma, which shows unique solvability of the ODE \eqref{eq:psi}
and \eqref{eq:phi_k} for controls $\boldsymbol{\varphi}$ in a suitable class, is key in
the proof.

\begin{Lemma}
\label{lem:uniqueness} Fix $\sigma \in (0,1)$. Given $(\boldsymbol{\zeta} ,\psi ) \in 
\mathcal{C}_T$ with $I_T(\boldsymbol{\zeta} ,\psi ) < \infty$, there exists $\boldsymbol{\varphi}^* \in 
\mathcal{S}_T(\boldsymbol{\zeta} ,\psi )$ such that

\begin{enumerate}[\upshape(a)]

\item $\sum_{k=0}^\infty \int_{[0,T] \times [0,1]} \ell(\varphi_k^*(s,y)) \,
ds\,dy \le I_T(\boldsymbol{\zeta} ,\psi ) + \sigma$.

\item If $(\tilde{\boldsymbol{\zeta}},\tilde\psi)$ is another pair in $\mathcal{C}_T$ such
that $\boldsymbol{\varphi}^* \in \mathcal{S}_T(\tilde{\boldsymbol{\zeta}},\tilde\psi)$, then $%
(\tilde{\boldsymbol{\zeta}},\tilde\psi) = (\boldsymbol{\zeta} ,\psi )$.
\end{enumerate}

\end{Lemma}

\begin{proof}
	Since $I_T(\boldsymbol{\zeta} ,\psi ) < \infty$, we can choose some $\boldsymbol{\varphi} \in \mathcal{S}_T(\boldsymbol{\zeta} ,\psi )$ such that
	\begin{equation*}
		\sum_{k=0}^\infty \int_{[0,T] \times [0,1]} \ell(\varphi_k(s,y)) \, ds\,dy \le I_T(\boldsymbol{\zeta} ,\psi ) + \frac{\sigma}{2}.
	\end{equation*}
	Next we will modify $\boldsymbol{\varphi}$ to get the desired $\boldsymbol{\varphi}^*$.
	Since $\ell$ is convex and nonnegative and $\ell(1)=0$, we can assume without loss of generality that  for each $k \in \mathbb{N}_0$ and $(t,y) \in [0,T]\times[0,1]$, $\varphi_k(t,y)$ takes the form
	\begin{equation*}
		\varphi_k(t,y) = \rho_k(t) {{1}}_{[0,r_k(\boldsymbol{\zeta}(t)))}(y) + {{1}}_{[r_k(\boldsymbol{\zeta}(t)),1]}(y)
	\end{equation*}
	for some $\rho_k(t) \in [0,\infty)$.
	Fix $\varepsilon \in (0,1)$.  We will shrink the support of $\boldsymbol{\varphi}$ to get the desired $\boldsymbol{\varphi}^*$ for sufficiently small $\varepsilon$.
	For $t \in [0,T]$, let
	\begin{equation*}
		\varphi_k^\varepsilon(t,y) = \frac{\rho_k(t)}{1-\varepsilon} {{1}}_{[0,(1-\varepsilon)r_k(\boldsymbol{\zeta}(t)))}(y) + {{1}}_{[(1+\varepsilon)r_k(\boldsymbol{\zeta}(t)),1]}(y). 
	\end{equation*}
	It is clear that $\boldsymbol{\varphi}^\varepsilon \in \mathcal{S}_T(\boldsymbol{\zeta} ,\psi )$.	
	Note that $\varphi_k^\varepsilon(t,y) = 0$ for $(1-\varepsilon)r_k(\boldsymbol{\zeta}(t)) < y < (1+\varepsilon)r_k(\boldsymbol{\zeta}(t))$, which will be a key when we prove uniqueness in part (b).
	Recall $\tau_{\boldsymbol{\zeta}}$ introduced in \eqref{eq:defntauphi}.
	Then
	\begin{align*}
		& \sum_{k=0}^\infty \int_{[0,T] \times [0,1]} \ell(\varphi_k^\varepsilon(t,y)) \, dt\,dy - \sum_{k=0}^\infty \int_{[0,T] \times [0,1]} \ell(\varphi_k(t,y)) \, dt\,dy \\
		& = \sum_{k=0}^\infty \int_0^{\tau_{\boldsymbol{\zeta}}} \left[ (1-\varepsilon)r_k(\boldsymbol{\zeta}(t)) \ell(\frac{\rho_k(t)}{1-\varepsilon}) + 2\varepsilon r_k(\boldsymbol{\zeta}(t)) \ell(0) - r_k(\boldsymbol{\zeta}(t)) \ell(\rho_k(t)) \right] dt \\
		& = \sum_{k=0}^\infty \int_0^{\tau_{\boldsymbol{\zeta}}} r_k(\boldsymbol{\zeta}(t)) \left[ \left( \rho_k(t) \log (\frac{\rho_k(t)}{1-\varepsilon}) - \rho_k(t) + 1-\varepsilon \right) + 2\varepsilon \right. \\
		& \qquad \left. - \left( \rho_k(t) \log \rho_k(t) - \rho_k(t) + 1 \right) \right] dt \\
		& = \sum_{k=0}^\infty \int_0^{\tau_{\boldsymbol{\zeta}}} r_k(\boldsymbol{\zeta}(t)) \left[ \rho_k(t) \log (\frac{1}{1-\varepsilon}) + \varepsilon \right] dt.
	\end{align*}
	From Lemma \ref{lem:property_ell}(b) we have
	\begin{align*}
		& \sum_{k=0}^\infty \int_{[0,T] \times [0,1]} \ell(\varphi_k^\varepsilon(t,y)) \, dt\,dy - \sum_{k=0}^\infty \int_{[0,T] \times [0,1]} \ell(\varphi_k(t,y)) \, dt\,dy \\
		& \le \sum_{k=0}^\infty \int_0^{\tau_{\boldsymbol{\zeta}}} r_k(\boldsymbol{\zeta}(t)) \left[ \left( \ell(\rho_k(t)) + 2 \right) \log (\frac{1}{1-\varepsilon}) + \varepsilon \right] dt \\
		& = \log (\frac{1}{1-\varepsilon})\sum_{k=0}^\infty \int_{[0,T] \times [0,1]} \ell(\varphi_k(t,y)) \, dt\,dy  + 2 \tau_{\boldsymbol{\zeta}} \log (\frac{1}{1-\varepsilon}) + \tau_{\boldsymbol{\zeta}} \varepsilon \\
		& \le (I_T(\boldsymbol{\zeta} ,\psi ) + \frac{\sigma}{2})\log (\frac{1}{1-\varepsilon}) + 2T\log (\frac{1}{1-\varepsilon}) + T \varepsilon .
	\end{align*}	
	Choosing $\varepsilon$ small enough so that the last display is no larger than $\frac{\sigma}{2}$, we have
	\begin{align*}
		\sum_{k=0}^\infty \int_{[0,T] \times [0,1]} \ell(\varphi_k^\varepsilon(s,y)) \, ds\,dy \le \sum_{k=0}^\infty \int_{[0,T] \times [0,1]} \ell(\varphi_k(s,y)) \, ds\,dy + \frac{\sigma}{2} \le I_T(\boldsymbol{\zeta} ,\psi ) + \sigma.
	\end{align*}
	Part (a) then holds with $\boldsymbol{\varphi}^* = \boldsymbol{\varphi}^\varepsilon$ for  such an $\varepsilon$.
	
	We now show that  part (b) is satisfied with such a $\boldsymbol{\varphi}^*$. Suppose that, in addition to $(\boldsymbol{\zeta} ,\psi )$, there is another pair of $({\tilde{\boldsymbol{\zeta}}},{\tilde{\psi}})$ such that $({\tilde{\boldsymbol{\zeta}}},{\tilde{\psi}}) \in \mathcal{C}_T$ and $\boldsymbol{\varphi}^* \in \mathcal{S}_T({\tilde{\boldsymbol{\zeta}}},{\tilde{\psi}})$.
	Let 
	$$\tau \doteq \inf \{ t \in [0,T] : \boldsymbol{\zeta}(t) \ne {\tilde{\boldsymbol{\zeta}}}(t) \} \wedge T.$$
	We claim that $\tau = T$.
	Once the claim is verified, it follows from continuity of $\boldsymbol{\zeta}$ and ${\tilde{\boldsymbol{\zeta}}}$ that $\boldsymbol{\zeta}(t) = {\tilde{\boldsymbol{\zeta}}}(t)$ for all $t \in [0,T]$.
	Then from \eqref{eq:psi} we will have that $\psi = {\tilde{\psi}}$ proving part (b).
	
	Now we prove the claim that $\tau = T$. We will argue via contradiction.
	Suppose that $\tau < T$.
	To complete the proof, it suffices to reach the following contradiction
	\begin{equation}
		\label{eq:uniqueness_contradiction}
		\boldsymbol{\zeta}(t) = {\tilde{\boldsymbol{\zeta}}}(t), t \in [\tau,\tau+\delta] \mbox{ for some } \delta > 0.
	\end{equation}	
	From definition of $\tau$ and \eqref{eq:psi} it follows that $(\boldsymbol{\zeta}(t),r(\boldsymbol{\zeta}(t)),\psi(t)) = ({\tilde{\boldsymbol{\zeta}}}(t),r({\tilde{\boldsymbol{\zeta}}}(t)),{\tilde{\psi}}(t))$ for all $t < \tau$.
	From Remark \ref{rmk:property_ODE}(a) we have that $r(\boldsymbol{\zeta}(\cdot)), r({\tilde{\boldsymbol{\zeta}}}(\cdot)) \in \mathcal{C}$.
	Then by continuity, $(\boldsymbol{\zeta}(t),r(\boldsymbol{\zeta}(t)),\psi(t)) = ({\tilde{\boldsymbol{\zeta}}}(t),r({\tilde{\boldsymbol{\zeta}}}(t)),{\tilde{\psi}}(t))$ for all $t \le \tau$.
	If $r(\boldsymbol{\zeta}(\tau)) = r({\tilde{\boldsymbol{\zeta}}}(\tau)) = 0$, then from Remark \ref{rmk:property_ODE}(c) we have $\boldsymbol{\zeta}(t) = {\tilde{\boldsymbol{\zeta}}}(t) = \boldsymbol{0}$ for all $t \ge \tau$, which gives \eqref{eq:uniqueness_contradiction}.	
	Now we show \eqref{eq:uniqueness_contradiction} for the remaining case: $r(\boldsymbol{\zeta}(\tau)) = r({\tilde{\boldsymbol{\zeta}}}(\tau)) > 0$.
	For this, note that by continuity of $r(\boldsymbol{\zeta})$ and $r({\tilde{\boldsymbol{\zeta}}})$, there exists some $\delta > 0$ such that for all $t \in [\tau,\tau+\delta]$, 
	\begin{equation}
		\label{eq:uniqueness_key_fraction}
		r(\boldsymbol{\zeta}(t)) > 0, r({\tilde{\boldsymbol{\zeta}}}(t)) > 0, \left| \frac{r(\boldsymbol{\zeta}(t))}{r({\tilde{\boldsymbol{\zeta}}}(t))} - 1 \right| < \varepsilon,
	\end{equation}
	where $\varepsilon$ is as in part (a) and recall that $\boldsymbol{\varphi}^*=\boldsymbol{\varphi}^\varepsilon$.
	We will argue in two steps.\\
	
	Step $1$: We will prove that
	\begin{equation}
		\label{eq:step_1_claim}
		\zeta_k(t) = {\tilde{\zeta}}_k(t) \mbox{ for all } t \in [\tau,\tau+\delta], k \in \mathbb{N}.
	\end{equation}
	Suppose not, namely there exists $k\in \mathbb{N}$ such that
	$$\tau_k \doteq \inf \{t \in [\tau,\tau+\delta] : \zeta_k(t) \ne {\tilde{\zeta}}_k(t) \}\wedge T$$
	satisfies  $\tau \le \tau_k < \tau+\delta$.
	By continuity, we have $\zeta_k(t) = {\tilde{\zeta}}_k(t)$ for $t \le \tau_k$.
	Note that we must have $\zeta_k(\tau_k) = {\tilde{\zeta}}_k(\tau_k) > 0$, since otherwise $\zeta_k(\tau_k) = {\tilde{\zeta}}_k(\tau_k) = 0$ and from Remark \ref{rmk:property_ODE}(b) we see that $\zeta_k(t) = {\tilde{\zeta}}_k(t) = 0$ for all $t \ge \tau_k$, which contradicts the definition of $\tau_k$.
	From \eqref{eq:uniqueness_key_fraction} it then follows that
	\begin{align*}
		r_k(\boldsymbol{\zeta}(\tau_k)) & = \frac{k\zeta_k(\tau_k)}{r(\boldsymbol{\zeta}(\tau_k))} > 0, \\
		| r_k(\boldsymbol{\zeta}(\tau_k)) - r_k({\tilde{\boldsymbol{\zeta}}}(\tau_k)) | & = \left| \frac{k\zeta_k(\tau_k)}{r(\boldsymbol{\zeta}(\tau_k))} - \frac{k{\tilde{\zeta}}_k(\tau_k)}{r({\tilde{\boldsymbol{\zeta}}}(\tau_k))} \right| = \frac{k\zeta_k(\tau_k)}{r(\boldsymbol{\zeta}(\tau_k))} \left| 1 - \frac{r(\boldsymbol{\zeta}(\tau_k)}{r({\tilde{\boldsymbol{\zeta}}}(\tau_k))} \right| < \varepsilon r_k(\boldsymbol{\zeta}(\tau_k)).
	\end{align*}
	Note that the last inequality crucially uses the property that $r_k(\boldsymbol{\zeta}(\tau_k))>0$.

	Once more by continuity, there exists some $\delta_k>0$ such that last two inequalities hold for $t \in [\tau_k,\tau_k+\delta_k]$, namely $$r_k(\boldsymbol{\zeta}(t)) > 0, \: (1-\varepsilon)r_k(\boldsymbol{\zeta}(t)) < r_k({\tilde{\boldsymbol{\zeta}}}(t)) < (1+\varepsilon)r_k(\boldsymbol{\zeta}(t)).$$
	From construction of $\varphi^\varepsilon$, we see that for $t \in [\tau_k,\tau_k+\delta_k]$,
	\begin{equation*}
		\int_{(\tau_k,t] \times [0,1]} {{1}}_{[0,r_k({\tilde{\boldsymbol{\zeta}}}(s)))}(y) \varphi_k^\varepsilon(s,y) \, ds\,dy = \int_{(\tau_k,t] \times [0,1]}  {{1}}_{[0,r_k(\boldsymbol{\zeta}(s)))}(y) \varphi_k^\varepsilon(s,y) \, ds\,dy.
	\end{equation*}
	It then follows from \eqref{eq:phi_k} that $\zeta_k(t) = {\tilde{\zeta}}_k(t)$ for all $t \le \tau_k+\delta_k$.
	This contradicts the definition of $\tau_k$.
	Therefore \eqref{eq:step_1_claim} must hold.\\
	
	Step $2$: We will prove that
	\begin{equation}
		\label{eq:step_2_claim}
		\zeta_0(t) = {\tilde{\zeta}}_0(t) \mbox{ for all } t \in [\tau,\tau+\delta].
	\end{equation}
	Let $\eta(t) \doteq \zeta_0(t) - \psi(t)$ and ${\tilde{\eta}}(t) \doteq {\tilde{\zeta}}_0(t) - {\tilde{\psi}}(t)$.
	From properties of the Skorokhod map $\Gamma$ (see, e.g., \cite[Section 3.6.C]{KaratzasShreve1991brownian}), we have that
	\begin{align}
		& \eta(0) = 0, \eta(t) \mbox{ is non-decreasing and } \int_0^T \zeta_0(t) \, \eta(dt) = 0, \label{eq:step_2_eta} \\
		& {\tilde{\eta}}(0) = 0, {\tilde{\eta}}(t) \mbox{ is non-decreasing and } \int_0^T {\tilde{\zeta}}_0(t) \, {\tilde{\eta}}(dt) = 0 \label{eq:step_2_etatil}.
	\end{align}
	Now consider the function $[\zeta_0(t) - {\tilde{\zeta}}_0(t)]^2$.
	Since $\zeta_0,\psi,{\tilde{\zeta}}_0,{\tilde{\psi}}$ are absolutely continuous, we have for $t \in [\tau,\tau+\delta]$,
	\begin{align}
		& [\zeta_0(t) - {\tilde{\zeta}}_0(t)]^2 \notag \\
		& = [\zeta_0(\tau) - {\tilde{\zeta}}_0(\tau)]^2 + 2 \int_\tau^t (\zeta_0(s) - {\tilde{\zeta}}_0(s)) (\zeta_0'(s) - {\tilde{\zeta}}_0'(s)) \, ds \notag \\
		& = 2 \int_\tau^t (\zeta_0(s) - {\tilde{\zeta}}_0(s)) (\psi'(s) - {\tilde{\psi}}'(s)) \, ds + 2 \int_\tau^t (\zeta_0(s) - {\tilde{\zeta}}_0(s)) (\eta_0'(s) - {\tilde{\eta}}_0'(s)) \, ds. \label{eq:step_2_phi}
	\end{align}
	From \eqref{eq:psi} and \eqref{eq:phi_k} we see that for $t \in [\tau,\tau+\delta]$,
	\begin{align*}
		\psi(t) & = \sum_{k=1}^\infty (k-2) (p_k-\zeta_k(t)) - 2\int_{[0,t] \times [0,1]} {{1}}_{[0,r_0(\boldsymbol{\zeta}(s)))}(y) \, \varphi_0^\varepsilon(s,y) ds\,dy, \\
		{\tilde{\psi}}(t) & = \sum_{k=1}^\infty (k-2) (p_k-{\tilde{\zeta}}_k(t)) - 2\int_{[0,t] \times [0,1]} {{1}}_{[0,r_0({\tilde{\boldsymbol{\zeta}}}(s)))}(y) \, \varphi_0^\varepsilon(s,y) ds\,dy.
	\end{align*}
	Taking the difference of these two displays and using \eqref{eq:step_1_claim}, we have that for $t \in [\tau,\tau+\delta]$,
	\begin{equation}
		\psi(t) - {\tilde{\psi}}(t) = - 2 \int_{[0,t] \times [0,1]} \left( {{1}}_{[0,r_0(\boldsymbol{\zeta}(s)))}(y) - {{1}}_{[0,r_0({\tilde{\boldsymbol{\zeta}}}(s)))}(y) \right) \varphi_0^\varepsilon(s,y) \, ds\,dy. \label{eq:step_2_psi}
	\end{equation}	
	Since for each fixed $y \ge 0$ the function $x\mapsto \frac{x}{x+y}$ is non-decreasing on $(-y,\infty)$, we have from \eqref{eq:step_1_claim} and \eqref{eq:uniqueness_key_fraction} that if for some
	$t \in [\tau,\tau+\delta]$,  $\zeta_0(t) \ge {\tilde{\zeta}}_0(t)$, then
	\begin{align*}
		r_0(\boldsymbol{\zeta}(t)) &= \frac{\zeta_0(t)}{\zeta_0(t) + \sum_{k=1}^\infty k \zeta_k(t)} \\
		&= \frac{\zeta_0(t)}{\zeta_0(t) + \sum_{k=1}^\infty k {\tilde{\zeta}}_k(t)} \ge \frac{{\tilde{\zeta}}_0(t)}{{\tilde{\zeta}}_0(t) + \sum_{k=1}^\infty k {\tilde{\zeta}}_k(t)} = r_0({\tilde{\boldsymbol{\zeta}}}(t)).	
	\end{align*}
	Therefore for $t \in [\tau,\tau+\delta]$,
	\begin{equation*}
		{{1}}_{[0,r_0(\boldsymbol{\zeta}(t)))}(y) \ge {{1}}_{[0,r_0({\tilde{\boldsymbol{\zeta}}}(t)))}(y) \mbox{ when } \zeta_0(t) \ge {\tilde{\zeta}}_0(t).
	\end{equation*}
	and similarly
	\begin{equation*}
		{{1}}_{[0,r_0(\boldsymbol{\zeta}(t)))}(y) \le {{1}}_{[0,r_0({\tilde{\boldsymbol{\zeta}}}(t)))}(y) \mbox{ when } \zeta_0(t) \le {\tilde{\zeta}}_0(t).
	\end{equation*}
	Combining these two inequalities with \eqref{eq:step_2_psi}, we see that
	\begin{equation}
		\label{eq:step_2_mono_1}
		(\zeta_0(s) - {\tilde{\zeta}}_0(s)) (\psi'(s) - {\tilde{\psi}}'(s)) \le 0, \mbox{ a.e. } s \in [\tau,\tau+\delta].
	\end{equation}
	Next from \eqref{eq:step_2_eta} and \eqref{eq:step_2_etatil} we see that for $t \in [\tau,\tau+\delta]$,
	\begin{align*}
		\int_\tau^t {{1}}_{\{\zeta_0(s) > {\tilde{\zeta}}_0(s)\}} (\zeta_0(s) - {\tilde{\zeta}}_0(s)) (\eta_0'(s) - {\tilde{\eta}}_0'(s)) \, ds & \le \int_\tau^t {{1}}_{\{\zeta_0(s) > {\tilde{\zeta}}_0(s)\}} (\zeta_0(s) - {\tilde{\zeta}}_0(s)) \eta_0'(s) \, ds \\
		& \le \int_\tau^t {{1}}_{\{\zeta_0(s) > 0\}} \zeta_0(s) \, \eta_0(ds) \\
		& = 0,
	\end{align*}
	and similarly
	\begin{equation*}
		\int_\tau^t {{1}}_{\{\zeta_0(s) < {\tilde{\zeta}}_0(s)\}} (\zeta_0(s) - {\tilde{\zeta}}_0(s)) (\eta_0'(s) - {\tilde{\eta}}_0'(s)) \, ds \le 0.
	\end{equation*}
	Combining these two inequalities with \eqref{eq:step_2_mono_1} and \eqref{eq:step_2_phi}, we have for $t \in [\tau,\tau+\delta]$
	\begin{equation*}
		[\zeta_0(t) - {\tilde{\zeta}}_0(t)]^2 \le 0.
	\end{equation*} 
	Thus \eqref{eq:step_2_claim} holds.	
Combining \eqref{eq:step_1_claim} and \eqref{eq:step_2_claim} in Steps $1$ and $2$ gives \eqref{eq:uniqueness_contradiction} and completes the proof.	
\end{proof}

We can now complete the proof of the Laplace lower bound. Fix $h \in \mathbb{%
C}_b(\mathcal{D}_\infty \times \mathcal{D})$ and $\sigma \in (0,1)$. Fix
some $\sigma$-optimal $(\boldsymbol{\zeta}^*,\psi^*) \in \mathcal{C}_T$ with $%
I_T(\boldsymbol{\zeta}^*,\psi^*) < \infty$, namely 
\begin{equation*}
I_T(\boldsymbol{\zeta}^*,\psi^*) + h(\boldsymbol{\zeta}^*,\psi^*) \le \inf_{(\boldsymbol{\zeta} ,\psi ) \in \mathcal{D}%
_\infty \times \mathcal{D}} \left\{ I_T(\boldsymbol{\zeta} ,\psi ) + h(\boldsymbol{\zeta} ,\psi ) \right\} +
\sigma.
\end{equation*}
Let $\boldsymbol{\varphi}^* \in \mathcal{S}_T(\boldsymbol{\zeta}^*,\psi^*)$ be as in Lemma \ref%
{lem:uniqueness} (with $(\boldsymbol{\zeta},\psi)$ there replaced by $(\boldsymbol{\zeta}^*,\psi^*)$). 
For each $n \in \mathbb{N}$ and $(s,y) \in [0,T] \times [0,1]$, consider the
deterministic control 
\begin{align*}
\varphi^n_k(s,y) & \doteq \frac{1}{n} {{1}}_{\{\varphi^*_k(s,y)
\le \frac{1}{n}\}} + \varphi^*_k(s,y) {{1}}_{\{\frac{1}{n} <
\varphi^*_k(s,y) < n\}} + n {{1}}_{\{\varphi^*_k(s,y) \ge n\}}, k
\le n, \\
\varphi^n_k(s,y) & \doteq 1, k > n.
\end{align*}
Then $\boldsymbol{\varphi}^n \doteq (\varphi^n_k) \in \bar{\mathcal{A}}_b$ and from %
\eqref{eq:mainrepn17} we have 
\begin{equation*}
-\frac{1}{n} \log {{E}} e^{-nh(\boldsymbol{X}^{n},Y^{n})} \le {{E}} \left\{
\sum_{k=0}^\infty \int_{[0,T] \times [0,1]} \ell(\varphi_k^n(s,y)) \, ds\,dy
+ h({\bar{\boldsymbol{X}}}^{n},{\bar{Y}}^{n}) \right\},
\end{equation*}
where $({\bar{\boldsymbol{X}}}^{n}, {\bar{Y}}^{n})$ are given as in %
\eqref{eq:Ybar_n_upper_temp}--\eqref{eq:Xbar_n_k_upper_temp}. Noting that
for all $n \in \mathbb{N}$, $k \in \mathbb{N}_0$ and $(s,y) \in [0,T] \times
[0,1]$, $\ell(\varphi^n_k(s,y)) \le \ell(\varphi^*_k(s,y))$, we have from
Lemma \ref{lem:uniqueness}(a) that \eqref{eq:cost_bd_upper} holds with $M_0$
replaced by $I_T(\boldsymbol{\zeta}^*,\psi^*) + 1$. Define $\{\bar{\boldsymbol{\nu}}^{n}\}$ as in \eqref{eq:nu_n_upper} with controls $\boldsymbol{\varphi}^n$. From Lemma \ref{lem:tightness} it follows that $%
\{(\bar{\boldsymbol{\nu}}^{n},{\bar{\boldsymbol{X}}}^{n},{\bar{Y}}^{n})\}$ is
tight. Assume without loss of generality that $(\bar{\boldsymbol{\nu}}^{n},{\bar{\boldsymbol{X}}}^{n},{\bar{Y}}^{n})$ converges along the whole sequence
weakly to $(\bar{\boldsymbol{\nu}},{\bar{\boldsymbol{X}}},{\bar{Y}})$, given on some probability space $%
(\Omega^*,\mathcal{F}^*,{P}^*)$. From the construction of $%
\boldsymbol{\varphi}^n$ we must have $\bar{\boldsymbol{\nu}} = \bar{\boldsymbol{\nu}}^{\boldsymbol{\varphi}^*}$ a.s.\ ${P}^*$, where $\bar{\boldsymbol{\nu}}^{\boldsymbol{\varphi}^*}$ is as defined in \eqref{eq:nu_n_upper} using $\boldsymbol{\varphi}^*$.
By Lemma \ref{lem:char_limit} we have $({\bar{\boldsymbol{X}}},{\bar{Y}}) \in \mathcal{C}%
_T$ and $\boldsymbol{\varphi}^* \in \mathcal{S}_T({\bar{\boldsymbol{X}}},{\bar{Y}})$ a.s.\ ${%
P}^*$. From Lemma \ref{lem:uniqueness}(b) it now follows that $(%
{\bar{\boldsymbol{X}}},{\bar{Y}}) = (\boldsymbol{\zeta}^*,\psi^*)$ a.s.\ ${P}^*$. Finally,
from Lemma \ref{lem:uniqueness}(a), 
\begin{align*}
\limsup_{n \to \infty} -\frac{1}{n} \log {{E}} e^{-nh(\boldsymbol{X}^{n},Y^{n})} & \le
\limsup_{n \to \infty} {{E}} \left\{ \sum_{k=0}^\infty \int_{[0,T]
\times [0,1]} \ell(\varphi_k^n(s,y)) \, ds\,dy + h({\bar{\boldsymbol{X}}}^n,{\bar{Y}}^n)
\right\} \\
& \le \sum_{k=0}^\infty \int_{[0,T] \times [0,1]} \ell(\varphi_k^*(s,y)) \,
ds\,dy + {{E}}^* h({\bar{\boldsymbol{X}}}, \bar Y) \\
& = \sum_{k=0}^\infty \int_{[0,T] \times [0,1]} \ell(\varphi_k^*(s,y)) \,
ds\,dy + h(\boldsymbol{\zeta}^*,\psi^*) \\
& \le I_T(\boldsymbol{\zeta}^*,\psi^*) + h(\boldsymbol{\zeta}^*,\psi^*) + \sigma \\
& \le \inf_{(\boldsymbol{\zeta} ,\psi ) \in \mathcal{D}_\infty \times \mathcal{D}} \left\{
I_T(\boldsymbol{\zeta} ,\psi ) + h(\boldsymbol{\zeta} ,\psi ) \right\} + 2\sigma.
\end{align*}
Since $\sigma \in (0,1)$ is arbitrary, this completes the proof of the
Laplace lower bound.

\section{Compact Sub-level Sets}

\label{sec:rate_function}

In this section we prove that the function $I_T$ defined in %
\eqref{eq:rate_function} is a rate function, namely the set $\Gamma_N \doteq
\{ (\boldsymbol{\zeta} ,\psi ) \in \mathcal{D}_\infty \times \mathcal{D} : I_T(\boldsymbol{\zeta} ,\psi )
\le N \}$ is compact for each fixed $N \in [0,\infty)$.

Take any sequence $\{(\boldsymbol{\zeta}^n,\psi^n)_{n \in \mathbb{N}}\} \subset \Gamma_N$.
Then $(\boldsymbol{\zeta}^n,\psi^n) \in \mathcal{C}_T$ and there exists some $\frac{1}{n}$%
-optimal $\boldsymbol{\varphi}^n \in \mathcal{S}_T(\boldsymbol{\zeta}^n,\psi^n)$, namely 
\begin{equation}  \label{eq:cost_bd_rate}
\sum_{k=0}^\infty \int_{[0,T] \times [0,1]} \ell(\varphi_k^n(s,y)) \, ds\,dy
\le I_T(\boldsymbol{\zeta}^n,\psi^n) + \frac{1}{n} \le N + \frac{1}{n}.
\end{equation}
Recalling \eqref{eq:psi} and \eqref{eq:phi_k} and letting $\eta^n(t) \doteq
\zeta^n_0(t) - \psi^n(t)$, we can write for $t \in [0,T]$, 
\begin{equation}
\zeta^n_0(t) = \Gamma(\psi^n) = \psi^n(t) + \eta^n(t) = \sum_{k=0}^\infty
(k-2) B_k^n(t) + \eta^n(t),  \label{eq:phi_n_0_rate}
\end{equation}
where 
\begin{equation}
B_k^n(t) \doteq \int_{[0,t] \times [0,1]} {{1}}%
_{[0,r_k(\boldsymbol{\zeta}^n(s)))}(y) \, \varphi_k^n(s,y) \,ds\,dy, k \in \mathbb{N}_0.
\label{eq:B_n_k_rate}
\end{equation}
From standard properties of one-dimensional Skorokhod Problem we have 
\begin{equation}
\eta^n(0) = 0, \eta^n(t) \mbox{ is non-decreasing and } \int_0^T {%
{1}}_{\{\zeta^n_0(t)>0\}} \, \eta^n(dt) = 0.  \label{eq:eta_n_rate}
\end{equation}

\begin{Lemma}
\label{lem:UI_rate} Define for $K \in \mathbb{N}$, 
\begin{equation*}
U_K \doteq \sup_{n \in \mathbb{N}} \sum_{k=K}^\infty \int_{[0,T] \times
[0,1]} k \varphi^n_k(s,y) {{1}}_{[0,r_k(\boldsymbol{\zeta}^n(s)))}(y) \,ds\,dy.
\end{equation*}
Then as $K \to \infty$, 
\begin{equation*}
(U_K, \sup_{n \in \mathbb{N}} \sum_{k=K}^\infty k \|B^n_k\|_\infty, \sup_{n
\in \mathbb{N}} \sum_{k=K}^\infty k \|\zeta^n_k\|_\infty) \to {\boldsymbol{0}}%
.
\end{equation*}
\end{Lemma}

\begin{proof}
	From \eqref{eq:B_n_k_rate} and \eqref{eq:phi_k} (applied to $\boldsymbol{\zeta}^n$) it follows that for $K \in \mathbb{N}$,
	\begin{align*}
		& U_K = \sup_{n \in \mathbb{N}} \sum_{k=K}^\infty k B^n_k(T) = \sup_{n \in \mathbb{N}} \sum_{k=K}^\infty k \left( p_k - \zeta^n_k(T) \right) \le \sum_{k=K}^\infty k p_k, \\
		& \sup_{n \in \mathbb{N}} \sum_{k=K}^\infty k \|B^n_k\|_\infty \le \sup_{n \in \mathbb{N}} \sum_{k=K}^\infty k \|\zeta^n_k\|_\infty = \sum_{k=K}^\infty k p_k.
	\end{align*}
	The result then follow from Assumptions \ref{asp:convgN} and \ref{asp:exponential-boundN}.
\end{proof}

Write $\boldsymbol{B}^n = (B^n_k)_{n \in \mathbb{N}_0}$ and let $\boldsymbol{\nu}^n$ be
defined as in \eqref{eq:nu_n_upper} with deterministic controls $\boldsymbol{\varphi}^n$. The following lemma shows that $%
\{(\boldsymbol{\nu}^n,\boldsymbol{\zeta}^n,\psi^n,\boldsymbol{B}^n,\eta^n)\}$ is pre-compact. The proof is
similar to that of Lemma \ref{lem:tightness} so we only provide a sketch
here.

\begin{Lemma}
\label{lem:pre_compact_rate} $\{(\boldsymbol{\nu}^n,\boldsymbol{\zeta}^n,\psi^n,\boldsymbol{B}^n,\eta^n)\} 
$ is pre-compact in $[\mathcal{M}_{FC}([0,T]\times[0,1])]^\infty \times 
\mathcal{C}_\infty \times \mathcal{C} \times \mathcal{C}_\infty \times 
\mathcal{C}$.
\end{Lemma}

\begin{proof}
	It suffices to argue pre-compactness of each coordinate.
	From \eqref{eq:cost_bd_rate} it follows that \eqref{eq:cost_bd_upper} holds with $K_0$ replaced by $N+1$.
	Then as in the proof of Lemma \ref{lem:tightness} we have that $\{\nu^{\varphi^n}\}$ is pre-compact in $[\mathcal{M}([0,T]\times[0,1])]^\infty$.
	
	Next we consider $\{(\psi^n,\boldsymbol{B}^n)\}$.
	Since $\psi^n(0) = 0$ and $B^n_k(0) = 0$ for each $k \in \mathbb{N}_0$, in order to argue pre-compactness of $\{(\psi^n,\boldsymbol{B}^n)\}$, it suffices to show that
	\begin{align*}
		& \limsup_{\delta \to 0} \limsup_{n \to \infty} \sup_{|t-s|\le\delta} |\psi^n(t) - \psi^n(s)| = 0, \\
		& \limsup_{\delta \to 0} \limsup_{n \to \infty} \sup_{|t-s|\le\delta} |B^n_k(t) - B^n_k(s)| = 0, k \in \mathbb{N}_0. 
	\end{align*}
	From \eqref{eq:psi} (applied to $\boldsymbol{\zeta}^n$) and \eqref{eq:B_n_k_rate} we have for $0 \le t-s \le \delta$,
	\begin{equation*}
		|\psi^n(t) - \psi^n(s)| \le \sum_{k=0}^\infty (k+2) |B^n_k(t) - B^n_k(s)|.
	\end{equation*}
	It then suffices to show
	\begin{equation}
		\label{eq:pre_compact_B_rate}
		\limsup_{\delta \to 0} \limsup_{n \to \infty} \sup_{|t-s|\le\delta} \sum_{k=0}^\infty (k+2) |B^n_k(t) - B^n_k(s)| = 0.
	\end{equation}
	Note that for  $0 \le t-s \le \delta$
	\begin{equation*}
		\sum_{k=0}^\infty (k+2) |B^n_k(t) - B^n_k(s)| = \sum_{k=0}^\infty \int_{(s,t] \times [0,1]} (k+2) \varphi^n_k(u,y)  {{1}}_{[0,r_k(\boldsymbol{\zeta}^n(u)))}(y) \,du\,dy. 
	\end{equation*}
	For every $K \in \mathbb{N}$, $M \in (0,\infty)$, the last expression can be bounded above by
	\begin{align*}
		& \sum_{k=0}^{K-1} \left[ \int_{(s,t] \times [0,1]} (k+2) \varphi^n_k(u,y) {{1}}_{\{ \varphi^n_k(u,y) > M \}}  {{1}}_{[0,r_k(\boldsymbol{\zeta}^n(u)))}(y) \,du\,dy \right. \\
		& \qquad \left. + \int_{(s,t] \times [0,1]} (k+2) \varphi^n_k(u,y) {{1}}_{\{ \varphi^n_k(u,y) \le M \}}  {{1}}_{[0,r_k(\boldsymbol{\zeta}^n(u)))}(y) \,du\,dy \right] + 3U_K \\
		& \le \sum_{k=0}^{K-1} \int_{(s,t] \times [0,1]} (K+1) \gamma(M) \ell(\varphi^n_k(u,y)) \,du\,dy + K(K+1)M\delta + 3U_K \\
		& \le (K+1)\gamma(M)(N+1) + K(K+1)M\delta + 3 U_K,
	\end{align*}
	where the first inequality follows from Lemma \ref{lem:property_ell}(a) and the last inequality follows from \eqref{eq:cost_bd_rate}.
	Therefore
	\begin{equation*}
		\limsup_{\delta \to 0} \limsup_{n \to \infty} \sup_{|t-s|\le\delta} |\psi^n(t) - \psi^n(s)| \le (K+1)\gamma(M)(N+1) + 3 U_K.
	\end{equation*}
	Taking $M \to \infty$ and then $K \to \infty$, we have from Lemma \ref{lem:property_ell}(a) and Lemma \ref{lem:UI_rate} that \eqref{eq:pre_compact_B_rate} holds.
	So $\{(\psi^n,\boldsymbol{B}^n)\}$ is pre-compact in $\mathcal{C} \times \mathcal{C}_\infty$.	
	
	Finally we consider $\{(\boldsymbol{\zeta}^n,\eta^n)\}$.
	Since we have shown that $\{(\psi^n,\boldsymbol{B}^n)\}$ is pre-compact,  we have from the relation $\zeta^n_k = p_k - B^n_k$, the pre-compactness of $\{\zeta^n_k\}$ in $\mathcal{C}$ for each $k \in \mathbb{N}$.
	Since  $\zeta^n_0 = \Gamma(\psi^n) = \psi^n + \eta^n$, we have pre-compactness of $\{(\zeta^n_0, \eta^n)\}$ in $\mathcal{C}\times \mathcal{C}$.
	This completes the proof.	
\end{proof}

The following lemma characterizes limit points of $(\boldsymbol{\nu}^n,\boldsymbol{\zeta}^n,%
\psi^n,\boldsymbol{B}^n,\eta^n)$. Much of the proof is similar to that of Lemma \ref{lem:char_limit} except the proof of \eqref{eq:B_k_rate} for $k=0$.

\begin{Lemma}
\label{lem:cvg_rate} Suppose $(\boldsymbol{\nu}^n,\boldsymbol{\zeta}^n,\psi^n,\boldsymbol{B}^n,\eta^n)$
converges along a subsequence to $(\boldsymbol{\nu},\boldsymbol{\zeta},\psi,\boldsymbol{B},\eta) \in [\mathcal{M}%
([0,T]\times[0,1])]^\infty \times \mathcal{C}_\infty \times \mathcal{C}
\times \mathcal{C}_\infty \times \mathcal{C}$. Then the following holds.

\begin{enumerate}[\upshape(a)]

\item For each $k \in \mathbb{N}_0$, $\nu_k \ll \lambda_T$, and letting $%
\varphi_k \doteq \frac{d\nu_k}{d\lambda_T}$, 
\begin{equation*}
\sum_{k=0}^\infty \int_{[0,T] \times [0,1]} \ell(\varphi_k(s,y)) \, ds\,dy
\le N.
\end{equation*}

\item For each $t\in[0,T]$, 
\begin{align*}  \label{eq:phi_k_rate}
\zeta_0(t) & = \Gamma(\psi)(t) = \psi(t) + \eta(t), \\
\\
\zeta_k(t) & = p_k - B_k(t), \; k \in \mathbb{N}, \\
\psi(t) & = \sum_{k=0}^\infty (k-2) B_k(t).
\end{align*}

\item For each $t\in[0,T]$, 
\begin{equation}
B_k(t) = \int_{[0,t] \times [0,1]} {{1}}_{[0,r_k(\boldsymbol{\zeta}(s)))}(y)
\, \varphi_k(s,y) ds\,dy, k \in \mathbb{N}_0,  \label{eq:B_k_rate}
\end{equation}
in particular $(\boldsymbol{\zeta}, \psi) \in \mathcal{C}_T$ and $\boldsymbol{\varphi} \in \mathcal{S}%
_T(\boldsymbol{\zeta}, \psi)$.
\end{enumerate}
\end{Lemma}

\begin{proof}
	Assume without loss of generality that
	\begin{equation}
		\label{eq:cvg_rate_joint}
		(\boldsymbol{\nu}^n,\boldsymbol{\zeta}^n,\psi^n,\boldsymbol{B}^n,\eta^n) \to (\boldsymbol{\nu},\boldsymbol{\zeta},\psi,\boldsymbol{B},\eta)
	\end{equation} 
	as $n \to \infty$ along the whole sequence.
	
	(a) This is an immediate consequence of the bound in \eqref{eq:cost_bd_rate} and Lemma A.1 of \cite{BudhirajaChenDupuis2013large}.

	(b)
	Using \eqref{eq:cvg_rate_joint} we see that the first equation follows from \eqref{eq:phi_n_0_rate}
	and continuity of $\Gamma$, the second equation follows from the relation $\zeta_k^n = p_k- B_k^n$, and the last equation follows from the equality $\psi^n = \sum_{k=0}^\infty (k-2) B_k^n$ and Lemma \ref{lem:UI_rate}.
	
	(c)
	From \eqref{eq:cvg_rate_joint} and Lemma \ref{lem:UI_rate} we have
	\begin{equation*}
		r(\boldsymbol{\zeta}^n(t)) = (\zeta^n_0(t))^+ + \sum_{k=1}^\infty k \zeta^n_k(t) \to (\zeta_0(t))^+ + \sum_{k=1}^\infty k \zeta_k(t) = r(\boldsymbol{\zeta}(t))
	\end{equation*}
	uniformly in $t \in [0,T]$ as $n \to \infty$.
	Therefore $r(\boldsymbol{\zeta}(\cdot))$ is continuous.
	Let $\tau \doteq \inf \{ t \in [0,T] : r(\boldsymbol{\zeta}(t)) = 0\} \wedge T$.
	We will argue that \eqref{eq:B_k_rate} holds for all $t < \tau$, $t = \tau$ and $t > \tau$. The proof is similar to that of \eqref{eq:Bbar_k_upper}.
	
	For $t < \tau$, we have $r(\boldsymbol{\zeta}(t)) > 0$.
	Hence, as in proof of \eqref{eq:cgceofindic}, for each $k \in \mathbb{N}_0$,
	\begin{equation*}
		 {{1}}_{[0,r_k(\boldsymbol{\zeta}^n(s)))}(y) \to  {{1}}_{[0,r_k(\boldsymbol{\zeta}(s)))}(y)
	\end{equation*}
	as $n \to \infty$ for $\lambda_t$-a.e.\ $(s,y) \in [0,t] \times [0,1]$.
	Now using \eqref{eq:cost_bd_rate}, we have exactly as in the proof of \eqref{eq:cgcebhatkt} that
	as $n \to \infty$,
	\begin{align*}
		B^n_k(t) & \to \int_{[0,t] \times [0,1]} {{1}}_{[0,r_k(\boldsymbol{\zeta}(s)))}(y) \, \varphi_k(s,y) ds\,dy.
	\end{align*}
	So \eqref{eq:B_k_rate} holds for $t < \tau$.
	
	Since \eqref{eq:B_k_rate} holds for $t < \tau$, it also holds for $t = \tau$ by continuity of $B$ and of the right side in \eqref{eq:B_k_rate}.
	
	Now consider 	 $T\ge t > \tau$. Proof of the case $k \in \mathbb{N}$ is identical to the proof of \eqref{eq:Bbar_k_upper} for $k \in \mathbb{N}$ given below \eqref{eq:cgcebhatkt}.
	Next we show \eqref{eq:B_k_rate} for $k =0$.
	First note that from \eqref{eq:eta_n_rate} and \eqref{eq:phi_n_0_rate} we have for $\tau < t \le T$,
	\begin{equation*}
		|\eta^n(t) - \eta^n(\tau)| = \int_\tau^t \, \eta^n(ds) = \int_\tau^t {{1}}_{\{\zeta^n_0(s) = 0\}} \, \eta^n(ds) = \int_\tau^t {{1}}_{\{\zeta^n_0(s) = 0\}} \, (\zeta^n_0 - \sum_{k=0}^\infty (k-2)B^n_k)(ds).
	\end{equation*}
	From \eqref{eq:B_n_k_rate} we see that  $\int_\tau^t {{1}}_{\{\zeta^n_0(s) = 0\}} \, B^n_0(ds) = 0$.
	Also since $\zeta_0^n$ is non-negative and absolutely continuous, we have ${{1}}_{\{\zeta^n_0(s) = 0\}} (\zeta^n_0)'(s) = 0$ for a.e.\ $s \in [0,T]$.
	Therefore
	\begin{equation*}
		|\eta^n(t) - \eta^n(\tau)| 
		\le \sum_{k=1}^\infty |k-2| |B^n_k(t) - B^n_k(\tau)|.
	\end{equation*}
	Applying triangle inequality to \eqref{eq:phi_n_0_rate} and using this estimate, we see that
	\begin{align*}
		 \sup_{\tau < t \le T} |B^n_0(t) - B^n_0(\tau)| 
		 \le \sup_{\tau < t \le T} |\zeta^n_0(t) - \zeta^n_0(\tau)| + 2 \sum_{k=1}^\infty |k-2|\sup_{\tau < t \le T} |B^n_k(t) - B^n_k(\tau)|.
	\end{align*}
	Now as in the proof of \eqref{eq:middl} we have
	\begin{align*}
				\sup_{\tau < t \le T} |B^n_0(t) - B^n_0(\tau)| & \le 
				4 r(\boldsymbol{\zeta}^n(\tau)),
	\end{align*}
	which converges to $4 r(\boldsymbol{\zeta}(\tau))=0$ as $n \to \infty$.
	Hence $B_0(t) = B_0(\tau)$ for $\tau < t \le T$ and this gives \eqref{eq:B_k_rate} for $k=0$.
	
	Since we have proved \eqref{eq:B_k_rate} for all $t < \tau$, $t = \tau$ and $t > \tau$, part (c) follows.
\end{proof}

\noindent\textbf{Proof of compact sub-level sets $\Gamma_M$:} Now we are
ready to prove that $\Gamma_M$ is compact for each fixed $M \in [0,\infty)$.
Recall $(\boldsymbol{\zeta}^n,\psi^n)$ introduced above %
\eqref{eq:cost_bd_rate} and $\boldsymbol{\nu}^n$ introduced above Lemma \ref{lem:pre_compact_rate}. From Lemma \ref{lem:pre_compact_rate} we have
pre-compactness of $\{(\boldsymbol{\nu}^n,\boldsymbol{\zeta}^n,\psi^n)\}$ in $[\mathcal{M}%
([0,T]\times[0,1])]^\infty \times \mathcal{C}_\infty \times \mathcal{C}$.
Assume without loss of generality that $(\boldsymbol{\nu}^n,\boldsymbol{\zeta}^n,\psi^n)$
converges along the whole sequence to some $(\boldsymbol{\nu},\boldsymbol{\zeta},\psi)$.
By Lemma \ref{lem:cvg_rate} $(\boldsymbol{\zeta},\psi) \in \mathcal{C}_T$ and $\boldsymbol{\nu}= \boldsymbol{\nu}^{\boldsymbol{\varphi}}$,
where for $k \in \mathbb{N}_0$, ${\nu}_k^{\boldsymbol{\varphi}}$ is as defined by the right side of  \eqref{eq:nu_n_upper} replacing ${\varphi}_k^n$ with ${\varphi}_k$,
and 
\begin{equation*}
I_T(\boldsymbol{\zeta} ,\psi ) \le \sum_{k=0}^\infty \int_{[0,T] \times [0,1]}
\ell(\varphi_k(s,y)) \, ds\,dy \le M.
\end{equation*}
Therefore $(\boldsymbol{\zeta} ,\psi ) \in \Gamma_M$ which proves that $\Gamma_M$ is
compact. This completes the proof that $I_T(\cdot)$ defined in %
\eqref{eq:rate_function} is a rate function.

\begin{Remark}
\label{rmk:unique_varphi} Suppose that for all $n \in \mathbb{N}$,  $(\boldsymbol{\zeta}^n,\psi^n) =
(\boldsymbol{\zeta} ,\psi )$ for some $(\boldsymbol{\zeta} ,\psi ) \in \mathcal{C}_T$ with $I_T(\boldsymbol{\zeta} ,\psi ) <
\infty$ and $M = I_T(\boldsymbol{\zeta} ,\psi )$. Then taking $\boldsymbol{\varphi}^n$  satisfying \eqref{eq:cost_bd_rate} (with $(\boldsymbol{\zeta}^n,\psi^n)$ replaced with $(\boldsymbol{\zeta} ,\psi )$),
we see from the above argument that there
exists some $\boldsymbol{\varphi} \in \mathcal{S}_T(\boldsymbol{\zeta} ,\psi )$ such that 
\begin{equation*}
I_T(\boldsymbol{\zeta} ,\psi ) \le \sum_{k=0}^\infty \int_{[0,T] \times [0,1]}
\ell(\varphi_k(s,y)) \, ds\,dy \le I_T(\boldsymbol{\zeta} ,\psi ),
\end{equation*}
namely $I_T(\boldsymbol{\zeta} ,\psi )$ is achieved at some $\boldsymbol{\varphi} \in \mathcal{S}%
_T(\boldsymbol{\zeta} ,\psi )$.
\end{Remark}

\vspace{\baselineskip}\noindent \textbf{Acknowledgement:} The research of SB was supported in part by the National Science Foundation (DMS-1606839), the National Science Foundation (DMS-1613072) and the Army Research Office (W911NF-17-1-0010).
The research of AB was supported in part by the National Science Foundation (DMS-1305120), the Army Research Office
(W911NF-14-1-0331) and DARPA (W911NF-15-2-0122).
The research of PD was supported in part by National Science Foundation (DMS-1317199) and DARPA (W911NF-15-2-0122). 
The research of RW was supported in part by DARPA (W911NF-15-2-0122).

\bibliographystyle{plain}

\textsc{\noindent S. Bhamidi and A. Budhiraja\newline
Department of Statistics and Operations Research\newline
University of North Carolina\newline
Chapel Hill, NC 27599, USA\newline
email: bhamidi@email.unc.edu, budhiraj@email.unc.edu \vspace{\baselineskip} }

\textsc{\noindent P. Dupuis and R. Wu\newline
Division of Applied Mathematics\newline
Brown University\newline
Providence, RI 02912, USA\newline
email: paul\_dupuis@brown.edu, ruoyu\_wu@brown.edu }

\end{document}